\documentclass{article}
\usepackage{tikz}
\usepackage{mathtools}
\usepackage{amsmath}
\usepackage{amsfonts}
\usepackage{amssymb}
\usepackage{amsthm}
\usepackage{mathrsfs}
\usepackage{manfnt}
\usepackage{xcolor}
\usepackage{tcolorbox}
\usepackage{hyperref}
\usepackage{enumitem}
\usepackage{upgreek}
\usepackage{scalerel}
\usepackage{graphicx,stackengine}

\theoremstyle{plain}
\newtheorem{theorem}{Theorem}[section]
\theoremstyle{definition}
\newtheorem{definition}[theorem]{Definition}

\theoremstyle{plain}
\newtheorem{lemma}[theorem]{Lemma}

\theoremstyle{remark}
\newtheorem{propparagraph}[theorem]{}

\theoremstyle{remark}
\newtheorem{example}[theorem]{Example}
\newtheorem{remark}[theorem]{Remark}
\theoremstyle{plain}
\newtheorem{introTheorem}{Theorem}

\newcommand{\BourbakiCA}[1]{\cite[]{}}
\newcommand{\BourbakiAlg}[1]{\cite[]{bourbaki1998algebra,bourbaki2013algebra}}
\newcommand{\BourbakiLieAlg}[1]{\cite[]{}}
\newcommand{\BourbakiSet}[1]{\cite[]{}}

\DeclareMathOperator{\Tan}{Tan}

\newcommand{\weight}{\mathbin{\vrule height 1.6ex depth 0pt width
		0.14ex\vrule height 0.14ex depth 0pt width 1.3ex}}

\DeclareMathOperator{\G}{{\bf G}}
\DeclareMathOperator{\Var}{{\bf V}}
\DeclareMathOperator{\RVar}{{\bf RV}}
\DeclareMathOperator{\IVar}{{\bf IV}}

\DeclareMathOperator{\dist}{dist}

\DeclareMathOperator{\dmn}{dmn}

\DeclareMathOperator{\spt}{spt}
\DeclareMathOperator{\D}{D}

\DeclareMathOperator*{\Clos}{Clos}

\DeclareMathOperator{\Int}{Int}
\DeclareMathOperator*{\bfalpha}{\boldsymbol{\alpha}}

\DeclareMathOperator*{\Density}{\boldsymbol{\rm \Theta}}
\DeclareMathOperator{\vVarOp}{{\rm\bf v}}

\DeclareMathOperator{\bfT}{{\bf T}}

\DeclareMathOperator{\bfh}{{\bf h}}

\DeclareMathOperator{\im}{im}

\DeclareMathOperator{\grad}{grad}

\newcommand{\CodeComment}[1]{}

\newcommand{\Addresses}{{\bigskip
  \footnotesize

YTL:
\textsc{Department of Mathematics, National Taiwan Normal University, No. 88, Sec.4, Ting-Chou Road, Taipei, 116059, Taiwan (R.O.C.)}\par\nopagebreak
  \textit{Email address:} \texttt{80940004S@ntnu.edu.tw}

\medskip

MW:
\textsc{Department of Mathematics, National Taiwan Normal University, No. 88, Sec.4, Ting-Chou Road, Taipei, 116059, Taiwan (R.O.C.)}\par\nopagebreak
  \textit{Email address:} \texttt{myles.workman@math.ntnu.edu.tw}
}}

\author{Yu Tong Liu and Myles Workman}
\title{The space-time-Grassmann measure of the Brakke flow}
\date{}

\begin{document}
\maketitle

\begin{abstract}
    For a $k$-dimensional Brakke flow on an open subset $U \subset \mathbf{R}^{n}$, over an open time interval $J$, we prove the existence of a canonical space-time-Grassmann measure $\lambda$, over $J \times \mathbf{G}_{k} (U)$, and give a characterisation of the flow with respect to the space-time weight of this measure.
    This results in a new definition of the Brakke flow, as that of a space-time measure which satisfies the Brakke inequality in a distributional sense.
    Each such space-time measure corresponds to a class of equivalent (classical) Brakke flows, thus yielding an equivalence between the classical definitions of the Brakke flow, and this new definition.
    Moreover, we prove that the mean curvature vector, density, and tangent map along the flow, are all measurable with respect to this space-time weight measure. 
\end{abstract}



\tableofcontents
\section{Introduction}

 The $k$-dimensional Brakke flow (introduced by Brakke \cite{BrakkeBook}) on an open subset $U$ of $\mathbf{R}^{n}$, over an open time interval $J$, is a family of Radon measures on $U$, which evolve as a weak, measure theoretic solution to the mean curvature flow (see \ref{Ilmanen-def-Brakke flow} for a definition). 
In this paper, given a Brakke flow, we prove the existence of a canonical space-time-Grassmann measure $\lambda$, on $J \times \mathbf{G}_{k} (U)$, associated to the flow, and give a characterisation of the flow with respect to the space-time weight $\| \lambda \|$, of this measure. 
The existence (and definition) of $\lambda$ is given in \textbf{Theorem \ref{intro:thm:2}}, and the characterisation of the Brakke flow with respect to $\| \lambda \|$, is summarized in \textbf{Theorem \ref{intro:thm:4}}, but first, we give some motivation behind this space-time-Grassmann measure, and this characterisation 

One motivation is that these measures yield a natural equivalence class of Brakke flows (\ref{Space time Brakke flow:4}), with each class having well defined left continuous (\ref{Space time Brakke flow:1}), and right continuous (\ref{Space time Brakke flow:2}), representative flows (\ref{Space time Brakke flow:3}). 
These left and right representatives characterise the jump discontinuities in the Brakke flow, which occur when these representative flows disagree. 
This arises from the fact that Brakke flows exhibit the behaviour of one dimensional functions of bounded variation (\ref{Space time Brakke flow:0}). 
Moreover, \textbf{Theorem \ref{intro:thm:4}} gives a new definition for the Brakke flow, as that of a space-time measure which satisfies the distributional inequality (\ref{eq:Brakke's distributional inequality}) (see \ref{def:Brakke flow:space time measure}).
We also prove that the mean curvature vector, density, and tangent map along the flow, are all measurable with respect to the space-time weight measure $\| \lambda \|$ (\textbf{Theorem \ref{intro:thm:3}}).
A further motivation, is that the formulation of these measures, and the properties that we have proven, lay the groundwork for results on integral varifolds with locally bounded first variation, to be appropriately adapted to this parabolic setting.

To prove the existence of this measure, we recall that for a Brakke flow $\{\mu (t) \colon t \in J \}$, on $U$, for $\mathscr{L}^{1}$ almost every $t \in J$, there exists a unique integral varifold $V (t) \in \mathbf{IV}_{k} (U)$, such that $\| V (t) \| = \mu (t)$.
We prove that the map $t \mapsto V (t)$ is $\mathscr{L}^{1}$ measurable on $J$ (\ref{Brakke flow:thm:1:6}), and thus, can define this space-time-Grassmann measure in the following way: 

\begin{introTheorem}[see \ref{thm:Brakke flow:space time measure}]\label{intro:thm:2}
Suppose 
$J$ is an open interval, $U$ is an open subset of $\mathbf R^n$, and
$\{\mu (t) \colon t \in J\}$ is a Brakke flow, with the associated family of integral varifolds, $\{V (t) \colon t \in \text{dmn} \, V \}$.
Then there exists a Radon measure $\lambda$ over $J \times \mathbf{G}_{k} (U)$, characterised by 
\[\textstyle  \lambda(f) = \int \int f(t,x,S) \ d V(t)_{(x,S)} \ d \mathscr L^1, \ \text{for $f\in \mathscr K(J \times \G_k(U))$}.\]
\end{introTheorem}

Upon defining the Radon measure $\lambda$, we then 
associate the space-time weight measure $\|\lambda\|(f) = \int  f(t,x) \ d \lambda_{(t,x,S)}$ for $f\in \mathscr K(J\times U)$.
We now state the previously mentioned characterisation.

\begin{introTheorem}\label{intro:thm:4}\label{thm:Space time Brakke flow}
    Suppose $J \subset \mathbf R$ is a bounded open interval, 
    $U$ is an open subset of $\mathbf R^n$, 
    and $V$ is a Radon measure over $J\times \G_k(U)$. 
    For $\mathscr{L}^{1}$ almost all $t \in J$ we may define the following Radon measures over $\mathbf{G}_{k} (U)$,
    \begin{gather*}
        \langle V^+,t\rangle  (\alpha) = 
        \lim\limits_{\varepsilon \downarrow 0} \frac{1}{\varepsilon} \textstyle \int_{\{ (s,x,S) : t \leq s < t + \varepsilon \}} \alpha(x,S) \ d V_{(s,x,S)},\\
        \langle V^-,t\rangle  (\alpha) = 
        \lim\limits_{\varepsilon \downarrow 0} \frac{1}{\varepsilon} \textstyle \int_{\{ (s,x,S) : t -\varepsilon< s \leq  t  \}} \alpha(x,S) \ d V_{(s,x,S)},\\
        \langle V,t\rangle  (\alpha) = 
        \lim\limits_{\varepsilon \downarrow 0} \frac{1}{2\varepsilon} \textstyle \int_{\{ (s,x,S) : t -\varepsilon< s \leq  t + \varepsilon \}} \alpha(x,S) \ d V_{(s,x,S)},
    \end{gather*}
    for $\alpha \in \mathscr K(\G_k(U))$.
    Now further suppose that, for every non-negative $\omega \in \mathscr D(J)$, and every non-negative, class $2$ $\phi :  U \to \mathbf R$, with compact support, we have that
    \begin{equation}\label{eq:Brakke's distributional inequality}
        - \textstyle \int \omega'(t) \phi(x) \ d \|V\|_{(t,x)} 
        \leq \textstyle \int \omega(t) \mathscr B(\|\langle V,t\rangle \|,\phi) \ d \mathscr L^1_t \in \mathbf R. 
    \end{equation}
    Then the following five statements hold:
    \begin{enumerate}[label=\ref{thm:Space time Brakke flow}.\arabic{enumi}]
        \item \label{Space time Brakke flow:0}
        For every non-negative $\phi : U \to \mathbf R$ of class $2$ with compact support, 
        the function $\|\langle V,t\rangle \|(\phi)$ for $t\in J$ is of locally bounded variation. 
        Moreover, for every $\alpha \in \mathscr{K} (\mathbf{G}_{k} (U))$,
        and $\omega \in \mathscr K(J)$, 
        \[ V_{(t,x, S)}(\omega(t)\cdot \alpha(x, S)) = \textstyle \int \omega(t)\cdot \langle V,t\rangle (\alpha) \ d \mathscr L^1_t,\]
        
        \item \label{Space time Brakke flow:1} $\|\langle V^-,\cdot \rangle\|$ is left continuous, with domain $J$, and 
        \[  \lim\limits_{t \downarrow a} \|\langle V^-,t\rangle \| \leq \|\langle V^-,a\rangle \| \ \text{for $a\in J$}.\]
        \item \label{Space time Brakke flow:2}
        $\|\langle V^+,\cdot \rangle\|$ is right continuous, with domain $J$, and 
        \[ \lim\limits_{t \uparrow a} \|\langle V^+,t\rangle \| \geq \|\langle V^+,a\rangle \| \  \text{for $a\in J$}.\]
        \item \label{Space time Brakke flow:3}
        $\|\langle V^+,\cdot \rangle\| $ and $\|\langle V^-,\cdot \rangle\| $ are  Brakke flows, and $\|\langle V^-,\cdot\rangle \|\geq  \|\langle V^+,\cdot\rangle \|$. 
        \item \label{Space time Brakke flow:4} 
        For an $\mathscr L^1\weight J$ measurable function $\nu$, from $J$
        to the Radon measures over $U$ with
        \[  \textstyle  \|V\|(\theta) = \int \int \theta(t,x) \ d \nu(t)_x \ d \mathscr L^1_t \ \text{for $\theta \in \mathscr K(J\times U)$},\]
        to be a Brakke flow,
        it is necessary and sufficient that for every $t\in J$ 
        \[  \|\langle V^- ,t \rangle \| \geq \nu(t) \geq \|\langle V^+,t\rangle \| .\]
        If one of the equivalent conditions holds and 
        \begin{equation}
        \label{sufficient condition for extension}
            \textstyle  \sup \{ -\int \omega'(t)\cdot \|\langle V,t \rangle \| (\phi) \ d \mathscr L^1_t : \omega \in \mathscr D(J), \  \sup \im |\omega|\leq 1\}<\infty
        \end{equation}  
        for non-negative $\phi : U \to \mathbf R$ of class $2$ with compact support, then, denote $a = \inf J$ and $b = \sup J$,
        we may extend $\nu$ to endpoints $a$ and $b$ by 
        \[ \nu(b) = \lim\limits_{s\uparrow b} \nu(s), \ \nu(a) = \lim\limits_{s\downarrow a} \nu(s)  \]
        to obtain a Brakke flow.
    \end{enumerate}
\end{introTheorem}

As previously noted, this new definition of the Brakke flow arising from \textbf{Theorem \ref{intro:thm:4}} (see \ref{def:Brakke flow:space time measure}), is equivalent to the classical definition of \ref{Ilmanen-def-Brakke flow}.
Another alternate classical definition of the Brakke flow (see \ref{def:Brakke flow}.\ref{def pt: time integrated definition of Brakke flow}), is also commonly found in the literature (\cite{tonegawa2019brakke}).
Lahiri (\cite{Lahiri2017EqualityOT}), has shown the equivalence of these two classical definitions (see also Ambrosio and Soner \cite[4.4]{AS-GeneralisedBrakkeFlow}), with a crucial aspect of the proof being the upper semicontinuity of the Brakke variation \cite[2.7]{Lahiri2017EqualityOT} (with the method of proof originating from Ilmanen \cite[7.3]{Ilmanen-EllipticRegularization}).
The method of proof in \cite[7.3]{Ilmanen-EllipticRegularization} can be summarised as consisting of two steps. 
Fixing a non-negative class 2 function $\phi$, with compact support in $U$, the first step in \cite[7.3]{Ilmanen-EllipticRegularization} is to prove the upper semicontinuity for non-negative, class 2 functions with compact support in $\{x \colon \phi (x) > 0\}$, and the second step is to extend this to $\phi$ by means of an appropriate approximation. 
However, as noted in \ref{Ilmanen-nonexistence-of-approximation}, such an approximation does not in general exist. 
In \ref{prop:interior convergence}, we give an alternate proof of this upper semicontinuity, which avoids the need for this approximation. 
Once we have this upper semicontinuity, the equivalence of these two classical definitions of the Brakke flow follows from arguments contained in \cite{Lahiri2017EqualityOT} (see \ref{def:Brakke flow}).
This upper semicontinuity is also crucial in the compactness theory for Brakke flows (\cite[7.1]{Ilmanen-EllipticRegularization}), as well as in this current paper, in proving the $\mathscr{L}^{1}$ measurability of the family of integral varifolds along the flow (\ref{Brakke flow:thm:1:6}). 

As previously mentioned, we also prove that the tangent map, mean curvature vector, and density are all $\|\lambda\|$ measurable along the flow.

\begin{introTheorem}[see \ref{thm:measurable function:tangent,density,mean curvature vector}]\label{intro:thm:3}
    
    Whenever $\lambda$ and $V$ are as in {\bf Theorem} \ref{intro:thm:2}, the function mapping $(t,x)\in J \times U$ onto 
    \[\textstyle  (\Tan(V(t),x) , \bfh(V(t),x), \Density^k(\|V(t)\|,x)) \in \G(n,k)\times \mathbf R^n\times \mathbf R\]
    is $\|\lambda\|$ measurable, with $\|\lambda\|$ measurable domain 
    covering $\|\lambda\|$ almost every point of $J\times U$,
    where $\Tan(V,a)$ denotes the unique tangent plane of $V$ at $a$ (see \ref{Unique tangent plane}).
\end{introTheorem}

The first step in this proof, is proving that these maps are Borel functions in the space $\mathbf{V}_{k} (U) \times U$, and thus, by composition with the map $(t, x) \mapsto (V (t), x)$, are $\| \lambda\|$ measurable. 
The Borel regularity follows from a standard technique, that is showing that the graph of the functions are Borel sets 
and then applying \cite[2.2.10]{MR41:1976}.

Finally, we note that formulations of a space-time-Grassmann measure, and its associated weight measure, for 
weak measure theoretic solutions to the mean curvature flow, have been explored previously.
In particular, Hensel and Laux (\cite{10.4310/jdg/1747065796})
have proposed a new weak solution concept for the mean curvature flow, which includes, in its definition, the existence of a space-time-Grassmann measure. 
The existence of such a flow was then proven in \cite{10.4310/jdg/1747065796} using an Allen-Cahn approximation. 
Moreover, Buet, Leonardi, Masnou and Sagueni (\cite{buet2025approximatemeancurvatureflows}), have recently utilised the idea of a space-time-Grassmann measure, in proving an existence result for the Brakke flow via a discrete time-step approximation scheme. 
 
 \section*{Acknowledgments}

Both authors are extremely grateful to Prof. Ulrich Menne, for sharing many ideas that have gone into this paper, in particular for the proof of \ref{prop:interior convergence}.
Both authors would also like to thank the anonymous referee for their feedback and comments.

\section{Notation}
    The  positive integers are denoted by $\mathscr P$.
    Unless otherwise stated, $n,k\in \mathscr P$ with $k\leq n$, and
    $U$ is an open subset of $\mathbf R^n$.
    The real numbers are denoted by $\mathbf R$, $\G(n,k)$ denotes the $k$ dimensional subspaces of $\mathbf R^n$.
    Suppose $X$ is a metric space, 
    $\mathbf U(a,r)$ and $\mathbf B(b,r)$ denote the open and closed ball respectively.
    Whenever $\phi$ measures a set $X$, 
    $A \subset X$, let $\phi\weight A( S) = \phi(A \cap S)$
    for $S \subset X$.
\begin{definition}\label{def:upper deriviatve, right and left upperderivative}
    Suppose $J$ is an interval, $f : J \to \mathbf R$, for every $t\in J$, we denote 
    \[ \D^-_s f(s)|_t= \limsup\limits_{s \uparrow t} \frac{f(s)-f(t)}{s-t}, \ \D^+_s f(s)|_t = \limsup\limits_{s \downarrow t} \frac{f(s)-f(t)}{s-t},\]
    and we denote their $\sup$ by $\overline{\D}_s f(s)|_t$, which equals
    \[ \limsup\limits_{s \to t}\frac{f(s)-f(t)}{s-t}.\]
    
    Whenever $X$ is a locally compact Hausdorff space,
    $K$ is a compact subset of $X$,
    $E$ is a normed space, 
    $\mathscr K(X,E)$ denotes the space of continuous functions from $X$ to $E$, with compact support,
    and $\mathscr K_K(X,E)$ the subspace of $f\in \mathscr K(X,E)$ with $\spt f \subset K$, and  let $\mathscr K(X) = \mathscr K(X,\mathbf R)$.
    We denote $\mathscr{K}' (X, E)$ to be the space of locally bounded linear functionals on $\mathscr{K} (X, E)$, i.e. $L \in \mathscr{K}' (X, E)$, if and only if, for all compact sets $K \subset X$, 
    \begin{equation*}
        \textstyle \sup \{ L (\phi) \, \colon \, \phi \in \mathscr{K}_{K} (X, E), \ \sup_{x \in E} |\phi (x)| \leq 1 \} < \infty.
    \end{equation*}

    Whenever $U$ is an open subset of $\mathbf R^n$ and $Y$ is a Banach space,
    $\mathscr D(U,Y)$ denotes the set of infinitely differentiable functions from $U$ to $Y$, with compact support, 
    $\mathscr D'(U,Y)$ denotes the corresponding distributions, and we define 
    \[ \boldsymbol{\nu}_K^i(\phi) = \sup \{ |\D^i\phi(x)| : x\in K\},\]
    for $\phi \in \mathscr D(U,Y)$, and compact subset $K$ of $U$.

\end{definition}

\begin{definition}[\protect{\cite[2.15]{Menne-WDFV}}]
      Suppose
      $U$ is an open subset of $\mathbf R^n$,
      $Y$ is a Banach space, and $T \in \mathscr D'(U,Y)$, then the  variation measure of $T$ is 
      the largest Borel regular measure 
      satisfying the following property, 
      \[ \|T\|(G) = \sup \{ T(\theta) : \theta \in \mathscr D(U,Y), \spt \theta \subset G, \sup \im \theta \leq 1\}\]
      for every open subset $G$ of $U$. 
      Whenever $Y$ is separable, 
      we say $T$ is representable by integration if $\|T\|$ is a Radon measure.
      In this case (see \cite[2.5.12]{MR41:1976}), 
      there exists a $\|T\|$ almost unique measurable function $\xi$ to 
      continuous linear functionals on $E$  
      such that for $\|T\|$ almost every $x$, 
      $|\xi(x)| = 1$,
           by the following property,
      \[\textstyle T(g) = \int \langle g(x) , \xi (x) \rangle \ d \|T\|_x \ \text{for $g\in  \mathscr D(U,Y)$}.\]
\end{definition}

\begin{definition}[\protect{\cite[3.5]{Allard1972OnTF}}]
\label{def:varifold}
    Suppose
    $U$ is an open subset of $\mathbf R^n$, then
    $\Var_k(U)$, $\RVar_k(U)$, and $\IVar_k(U)$ denote the sets of varifolds,
    rectifiable varifolds, and integral varifolds, respectively.
\end{definition}

  \begin{definition}\label{proppar:integral characterization of generalized mean curvature}
  Suppose $V \in \Var_k(U)$,
  $\delta V$ is representable by integration,
  then we define the function $\boldsymbol{\eta}(V,\cdot)$ to be characterised by
  \[z \bullet {\boldsymbol{\eta}(V,x)} = \lim\limits_{r\downarrow 0 } \frac{\delta V(b_{x,r}\cdot z)}{ \|\delta V\| (\mathbf B(x,r))} \ \text{whenever $z\in \mathbf R^n$ and the limit exists},
  \]
  and the generalized mean curvature $\bfh(V,\cdot)$ of $V$ is then characterised by
    \[ z \bullet \bfh(V,x) = \lim\limits_{r\downarrow 0} \frac{\delta V(  b_{x,r} \cdot z)}{\|V\|(\mathbf B(x,r))} \ \text{whenever $z\in \mathbf R^n$ and the limit exists},\]
 where $b_{x,r}$ is the characteristic function of $\mathbf B(x,r)$, note that 
 by \cite[2.5.12, 2.9.5]{MR41:1976}, $\|V\|(U \sim \dmn \bfh(V,\cdot)) + \|\delta V\|(U \sim \dmn \boldsymbol{\eta}(V,\cdot)) = 0$.
    \end{definition}

    \begin{definition}[\protect{\cite[3.4]{Allard1972OnTF}}]
    \label{Unique tangent plane}
        Whenever $\phi$ measures a metric space $X$,
        the $k$ dimensional upper and lower density is defined by 
        \[ \textstyle \Density^{k \ast }(\phi,a) = \limsup\limits_{r\downarrow 0}\frac{\phi(\mathbf B(a,r))}{\bfalpha(k)r^k}, \quad \Density^{k}_\ast(\phi,a) = \liminf\limits_{r\downarrow 0}\frac{\phi(\mathbf B(a,r))}{\bfalpha(k)r^k}\]
        where $\bfalpha(k) = \mathscr L^k(\mathbf B(0,1)) $,
        in case $\Density^{k \ast }(\phi,a) = \Density^{k}_\ast(\phi,a)$, their common value is denoted by 
        $\Density^k(\phi,a)$.
        A varifold $V \in \Var_k(U)$ 
        has unique tangent plane at $a\in U$, if 
        there exists $T\in \G(n,k)$
        such that $0<\Density^k(\|V\|,a)<\infty$,
        \[  \textstyle \Density^k(\|V\|,a)\cdot \int_T \varphi(x,T) \ d \mathscr H^k_x = \lim\limits_{r\downarrow 0}r^{-k} \int \varphi(r^{-1}(x-a),S) \ d V_{(x,S)}\]
        for 
        every $\varphi\in \mathscr K(\G_k(\mathbf R^n))$.
    \end{definition}

\section{Brakke flow}

\begin{definition}[\protect{\cite[6.3]{Ilmanen-EllipticRegularization}}]
\label{Ilmanen-def}
    Whenever 
    $U$ is an open subset of $\mathbf R^n$,
    $\mu$ is a Radon measure over $U$,
    $\phi : U \to \mathbf R$ is nonnegative
    of class $2$ with compact support,
    we denote $U_\phi$ to be the set of $x\in U$ such that $\phi(x)>0$, and $V_\phi$ the restriction of $V$ to subsets to $\G_k(U_\phi)$.
    If the following four conditions are satisfied 
    \begin{enumerate}[label=B-\arabic{enumi}]
        \item \label{integral}
         $\mu \weight U_\phi$ is a weight of some $V\in \IVar_k(U)$.
        \item \label{locally Radon}
        $\|\delta V_\phi\|$ is Radon measure over $U_{\phi}$.
        \item 
        \label{locally absolutely continuous}
     $\|\delta V_\phi\|$ is absolutely continuous with respect to $\|V_\phi\|$.
        \item 
        \label{weighted L2 mean curvature} 
        $\int \phi(x)\cdot |\bfh(V_\phi,x)|^2 \ d \|V_\phi\|_x<\infty$ ,
    \end{enumerate}
    then we define 
    \begin{equation}
        \label{regular case}\tag{{$\mathscr B$}}
        \textstyle \mathscr B(\mu,\phi)  = \int -\phi (x)\cdot |\bfh(V_\phi,x)|^2 +\grad \phi(x)\bullet \bfh(V_\phi,x)\ d \|V_\phi\|_x
    \end{equation} 
      Otherwise, we define
    $\mathscr B(\mu,\phi)=-\infty$.
\end{definition}
\begin{remark}\label{Locality of Brakke variation}
    The value of $\mathscr B(\mu,\phi)$ depends only on $\mu\weight U_\phi$.
\end{remark}
\begin{remark}\label{second order rectifiability} 
    If
    $\mu$ is a Radon measure over $U$,
    $-\infty < \mathscr B(\mu,\phi)$,
    then by \ref{integral}, there exists 
    $V \in \IVar_k(U)$ such that 
    $\|V\| = \mu \weight U_\phi$ and $\delta V_\phi$ is representable by integration, and by \cite[5.8]{BrakkeBook} or \cite[Theorem 1]{MR3023856},
    \begin{equation*}
        \bfh(V_\phi,x) \in S^\perp \ \text{for $V_\phi$ almost every $(x,S) \in \G_k(U_\phi)$},
        \end{equation*} 
    \begin{equation*}
        \mathscr B(\mu,\phi) = \textstyle \int -\phi(x) |\bfh(V_\phi,x)|^2 + S^\perp_\natural(\grad \phi (x))\bullet \bfh(V_\phi,x) \  d (V_{\phi})_{(x,S)}.
    \end{equation*}
\end{remark}

\begin{definition}\label{Ilmanen-def-Brakke flow}
   Suppose  $U$ is an open subset of $\mathbf{R}^{n}$, and $J$ is an interval of $\mathbf{R}$. 
    Then a function $\mu$ from $J$ to set of Radon measures over $U$ is a Brakke flow, if and only if for every  $s \in J$ and non-negative $\phi : U \to \mathbf R$ of class $2$ with compact support,
    \begin{equation*}
        \overline{\D}_{t} \mu  (t) (\phi)|_{s} \leq \mathscr{B} (\mu (s), \phi).
    \end{equation*}
\end{definition}

\begin{lemma}[\protect{\cite[2.2]{Lahiri2017EqualityOT}}]\label{proppar:1}
Suppose 
$V \in \IVar_k(U)$, 
    $\phi : U \to \mathbf R$ is a non-negative function of class $2$ with compact support, and
    $\|\delta V_\phi\|$ is a Radon measure.
Then the following two inequalities hold,
\begin{enumerate}[label=\ref{proppar:1}.\arabic{enumi}]
    \item \label{Elementary estimate:from schwartz ineq}
\[ \textstyle  \mathscr B(\|V\|,\phi)\leq 
        \int -\frac{\phi(x)}{2} \cdot |\bfh(V_\phi,x)|^2 + \frac{|\grad \phi(x)|^2}{ 2 \phi(x)} d \|V_\phi\|_x.\]
    \item \label{Elementary estimate:from diff eq} 
    \begin{align*}
          \mathscr B(\|V\|,\phi)\leq \textstyle \int
        - \D \grad \phi(x) \bullet S  \ d (V_{\phi})_{(x,S)}.
        \end{align*}
\end{enumerate}
  
    \end{lemma}
\begin{proof}
    We may assume
$-\infty < \mathscr B(\|V\|,\phi)$, from \ref{regular case} we have 
\[ \textstyle \mathscr B(\|V\|,\phi)  = \int -\phi (x)\cdot |\bfh(V_\phi,x)|^2 +\grad \phi(x)\bullet \bfh(V_\phi,x)\ d \|V_\phi\|_x.\]
By \cite[6.6]{Ilmanen-EllipticRegularization} and noting that $-\infty < \mathscr B(\|V \|,\phi)$, we have 
\[ \|V_\phi\|(\spt \grad \phi)+ \textstyle \int |\grad \phi| \ d \|\delta V_\phi\|<\infty,\]
    we may   apply \cite[9.3(2)]{menne2025prioriboundsgeodesicdiameter} with
        \begin{gather*}
        B = U \sim U_\phi, \
        \eta = \grad \phi, \ 
        W =V_\phi,
        \end{gather*}
        taking account of \cite[4.5]{10.2996/kmj/1521424824}, to obtain
        \[ \textstyle \int \D \grad \phi(x) \bullet S \ d (V_\phi)_{(x,S)}  = -\int \bfh(V_\phi,x) \bullet \grad \phi(x) \ d \|V_\phi\|_x.\]
    Then \ref{Elementary estimate:from diff eq} follows, and 
    \ref{Elementary estimate:from schwartz ineq} follows from Young's inequality.
 \end{proof}

\begin{propparagraph}\label{prop:interior convergence}
    \cite[7.3]{Ilmanen-EllipticRegularization}
    Suppose $U$ is an open subset of $\mathbf R^n$,
    $V_1,V_2,\dots,\in \Var_k (U)$, 
     $\phi : U \to \mathbf R$ is nonnegative of class $2$
     with compact support, 
     \begin{gather*}
          \limsup\limits_{i\to\infty}
    \|V_i\|(U_\phi)<\infty. 
    \end{gather*} 
    Then $V_1\weight \mathbf{G}_{k} (U_\phi),V_2\weight \mathbf{G}_{k} (U_\phi),\dots$ contains a convergent subsequence with limit   $V \in \mathbf{V}_{k} (U)$ such that 
    \begin{equation}\label{eq:upper semicontinuity}
        \limsup \limits_{i\to\infty }\mathscr B(\|V_i\|,\phi) \leq \mathscr B(\|V\|,\phi).
    \end{equation}  
    Moreover, if the left hand side of $\eqref{eq:upper semicontinuity}$ is not  $-\infty$, 
    and $V_1,V_2,\dots \in \IVar_k(U)$,
    then the restriction of $V$ to subsets of $\mathbf{G}_{k}(U_{\phi})$ is an integral varifold.

\end{propparagraph}

\begin{proof}
    By the compactness of the space of varifolds, there exists a convergent subsequence, and a limiting $V \in \mathbf{V}_{k} (U)$.
    If 
   the left hand side of \eqref{eq:upper semicontinuity} 
   is $-\infty$,
    then the statement follows trivially. 
    Thus, we may assume that 
    \begin{equation}\label{eq:reduction to limsup is not minus infinity}
        -\infty < \limsup\limits_{i\to\infty }\mathscr{B} (\| V _{i} \|, \phi) ,
    \end{equation} and after potentially taking a subsequence and renumerating, we assume that 
    the limit  in \eqref{eq:reduction to limsup is not minus infinity} is attained, and
    \begin{equation*}\label{regular condition is satisfied}
    -\infty < \mathscr{B} (\| V _{i} \|, \phi)  \ 
    \text{
    for  $i \in \mathscr P$},
    \end{equation*}
    \begin{equation}
    \label{existence of weak limit:V}
    V(\alpha) = \lim\limits_{i\to\infty} V_i\weight \mathbf{G}_{k} (U_\phi) (\alpha) \ \text{for $\alpha \in \mathscr K(\G_k(U))$}.
    \end{equation}
    By \ref{integral}, there exists $W_i \in \IVar_k(U)$ such that 
    \begin{equation}
    \label{conclusion of B1 for Vi , phi}
        \|W_i\| = \|V_i\| \weight U_\phi,
    \end{equation}
     we denote  $W_{i,\phi}$  the restriction of $W_i$ to subsets of ${\G_k(U_\phi)}$.
     Recalling \ref{regular case}, we have 
     \begin{equation}
     \label{eq:recall the formula of regular case for W i}
     \begin{aligned}
     \mathscr B(\|V_{i}\|,\phi) &= 
     \mathscr B(\|W_i\|,\phi) \\
     &= \textstyle \int - \phi (x) |\bfh(W_{i,\phi},x)|^2 + \grad \phi(x) \bullet \bfh(W_{i,\phi},x)\ d \|W_{i,\phi}\|_x
     \end{aligned}
     \end{equation}
By \ref{Elementary estimate:from schwartz ineq} and \cite[6.6]{Ilmanen-EllipticRegularization}, we estimate
    for every $i\in \mathscr P$,
    \begin{equation}
    \label{bounds on varaition of W i phi}
        \mathscr B(\|W_{i}\|,\phi) \leq - 2^{-1}  \textstyle  \int \phi (x) \cdot |\bfh(W_{i,\phi},x)|^2 \ d \|W_{i,\phi}\|_x \\
        + \|W_{i,\phi}\|(U_\phi) \cdot \boldsymbol{\nu}_{\spt \phi}^2(\phi).
    \end{equation}
    In view of \eqref{eq:reduction to limsup is not minus infinity},\eqref{conclusion of B1 for Vi , phi}, and \eqref{bounds on varaition of W i phi}, 
    we may apply \cite[6.4]{Allard1972OnTF}, 
    passing to another subsequence if necessary, 
    to infer 
     there exists $W\in \Var_k(U)$, denoting $W_\phi$ 
     the restriction of $W$ to subsets of $\G_k(U_\phi)$, such that 
     
    \begin{equation}\label{existence of weak limit : W}
        W(\alpha) = \lim\limits_{i\to\infty} W_i (\alpha) \ \text{for $\alpha \in \mathscr K(\G_k(U))$},
    \end{equation} 
    \begin{equation}\label{conclusion of integral compactness theorem: W phi}
        W_\phi(\alpha) = \lim\limits_{i\to\infty} W_{i,\phi}(\alpha) \ \text{for $\alpha \in \mathscr K(\G_k(U_\phi))$},
    \end{equation} 
    \begin{equation}\label{W weight U cap phi is integral}
         W \weight \G_k(U_\phi) \in \IVar_k(U).
    \end{equation}
    
    For every $f \in \mathscr K(U)$ with $\spt f \subset U_\phi$, by \eqref{existence of weak limit:V}, 
    \eqref{conclusion of B1 for Vi , phi}, and 
    \eqref{existence of weak limit : W},
    \[ \|W\|(f) = \lim\limits_{i\to\infty} \|V_i\|(f) = \|V\|(f),\]
    it follows that $\|W\|(G) = \|V\|(G)$ for every open subset $G$ of $U_\phi$, 
    therefore
    \begin{equation*}
        \|W \weight \G_k(U_\phi)\| = \|V\| \weight U_\phi.
    \end{equation*}
   Note that \eqref{bounds on varaition of W i phi} implies
   $\|\delta W_{\phi}\|$ is a Radon measure and absolutely continuous with respect to $\| W_{\phi} \|$, and 
   by duality (see \cite[Lemma A.3]{Lahiri2017EqualityOT}), 
    we have for either $\Tilde{W}$ equals $W_\phi$, or $W_{i,\phi}$
    \begin{equation*}
    \begin{aligned}
        \textstyle (\int &\phi(x) \cdot |\bfh(\Tilde{W},x)|^2 \ d \|\Tilde{W}\|_x  )^{1/2} \\
        & \textstyle = \sup \{ \int  g (x)\bullet \bfh(\Tilde{W},x)\cdot   \phi^{1/2}(x) \ d \|\Tilde{W}\|_x : g\in \mathscr D(U_\phi,\mathbf R^n),\ \textstyle \int |g|^2 \ d \|\Tilde{W}\|\leq 1\},
        \end{aligned}
        \end{equation*}
        and one then readily verifies by \eqref{conclusion of integral compactness theorem: W phi} that 
        \begin{equation}
        \label{lower semicontinuity}
        \liminf\limits_{i\to\infty} \textstyle \int \phi(x) \cdot |\bfh(W_{i,\phi},x)|^2 \ d \|W_{i,\phi}\|_x  \geq \int \phi (x)\cdot |\bfh(W_\phi,x)|^2 \ d \| W_\phi\|_x.
            \end{equation}
            Moreover, from
        \eqref{W weight U cap phi is integral} and 
        \eqref{lower semicontinuity}, we infer that       
    \begin{equation}\label{eq:W is regular in U phi}
        -\infty < \mathscr B(\|W\|,\phi) = \mathscr B(\|V\|,\phi),
    \end{equation} 
            and by \cite[6.6]{Ilmanen-EllipticRegularization}, we have that
            \[ \|W_{i,\phi}\|(U_\phi) + 
            \textstyle \int |\grad \phi| \ d \|\delta W_{i,\phi}\| <\infty \ \text{for $i\in \mathscr P$},\]
            \[ \|W_{\phi}\|(U_\phi) + 
            \textstyle \int |\grad \phi | \ d \|\delta W_{\phi}\| <\infty.\]
        Now,
        we may apply \cite[9.3(2)]{menne2025prioriboundsgeodesicdiameter} with
        \begin{gather*}
        B = U \sim U_\phi, \
        \eta = \grad \phi, \ 
        W = W_{i,\phi}, \ W_\phi 
        \end{gather*}
        taking account of \cite[4.5]{10.2996/kmj/1521424824}, to obtain
        \begin{equation}
            \label{eq:upper semicontinuity:Divergence formula}
            \textstyle \int \D \grad \phi(x) \bullet S \ d (W_{i,\phi})_{(x,S)}  = -\int \bfh(W_{i,\phi},x) \bullet \grad \phi(x) \ d \|W_{i,\phi}\|_x.
        \end{equation} 
        \begin{equation}
            \label{eq:upper semicontinuity:Divergence formula:W phi}
            \textstyle \int \D \grad \phi(x) \bullet S \ d (W_{\phi})_{(x,S)}  = -\int \bfh(W_{\phi},x) \bullet \grad \phi(x) \ d \|W_{\phi}\|_x.
        \end{equation} 
        From \eqref{conclusion of B1 for Vi , phi}, we have $\|W_i\|(U\sim U_\phi) = 0$, therefore
        \[\textstyle \int \D \grad \phi(x) \bullet S \ d (W_{i,\phi})_{(x,S)} = \int \D \grad \phi(x) \bullet S \ d (W_i)_{(x,S)},\]  
        and use \eqref{existence of weak limit : W}
        and
        \eqref{eq:upper semicontinuity:Divergence formula}
        to conclude that
        \begin{equation}
            \label{limit of delta W i phi}
           \textstyle  \lim\limits_{i\to\infty } -\int \bfh(W_{i,\phi},x) \bullet \grad \phi(x) \ d \|W_{i,\phi}\|_x = \int \D \grad \phi(x) \bullet S \ d W_{(x,S)}.
        \end{equation}         
        Observing that for $x\in U\sim U_\phi$ and $S\in \G(n,k)$, $\D \grad \phi(x) \bullet S \geq 0$, we have that
        \[\textstyle  \int \D \grad \phi(x) \bullet S \ d W_{(x,S)} \geq \int_{\G_k(U_\phi)} \D \grad \phi(x) \bullet S \ d W_{(x,S)}.\]
        The main conclusion follows from combining 
        \eqref{eq:recall the formula of regular case for W i},
        \eqref{lower semicontinuity}
        \eqref{eq:W is regular in U phi},
        \eqref{eq:upper semicontinuity:Divergence formula:W phi}, and
         \eqref{limit of delta W i phi}.
         In case $V_1,V_2,\ldots \in \IVar_k(U)$,
    for every  $i \in \mathscr P$, $W_{i,\phi}$ equals the restriction of $V_i$ to subsets of $\G_k(U_\phi)$, and the postscript follows. 
\end{proof}

\begin{remark}
    The proof of \ref{prop:interior convergence} follows the proof given in \cite[7.3]{Ilmanen-EllipticRegularization}, with a difference being that in \cite[7.3]{Ilmanen-EllipticRegularization}, it is claimed that
    \begin{equation*}
    \begin{aligned}
        \textstyle  \limsup\limits_{i\to\infty} \int \grad \phi (x) \bullet \bfh(W_{i,\phi},x) \ d \|W_{i,\phi}\|_x 
        \textstyle \leq \int \grad  \phi (x) \bullet \bfh(W_\phi,x) \ d \|W_\phi\|_x
    \end{aligned}
    \end{equation*}
    with the proof using an interior approximation of the function $\phi$. 
    However, example \ref{Ilmanen-nonexistence-of-approximation} shows that, in general, such an interior approximation does not hold for class 2 functions.
\end{remark}
\begin{example}\label{Ilmanen-nonexistence-of-approximation}
    We define 
    $g : \mathbf R \to \mathbf R$,  
    $g(t) = t^2$ near $0$, and $\spt g \subset \mathbf B(0,1)$.
One readily verifies that 
$g''(0) = 2$, and whenever $\psi \in \mathscr C^2(\mathbf R)$ with 
\begin{equation}\label{eq:1}
\spt \psi \subset \{ t : g(t)>0\}
\end{equation}
$\psi^{(k)}(0) = 0$.
 But for a function $\psi$ of class $2$ that is uniformly close to $g$ on $V= \{ t : 0 < |t|<\varepsilon\}$ for some $0<\varepsilon<\infty$,
 its second order derivative must have a lower bound away from $0$ on $V$, but $\psi^{(2)}(0) = 0$ would lead to a contradiction.
 \end{example}

 \begin{remark}
    The inequality
    \begin{equation*}
        \textstyle  \limsup\limits_{i\to\infty} \int \grad \phi (x) \bullet \bfh(W_{i,\phi},x) \ d \|W_{i,\phi}\|_x 
        \textstyle \leq \int \grad  \phi (x) \bullet \bfh(W_\phi,x) \ d \|W_\phi\|_x,
    \end{equation*}
    in the proof of \ref{prop:interior convergence} is sharp, in the sense that equality will not be achieved in general.
    In particular, consider $U = \mathbf{R}^{2}$, $W_{i} \in \mathbf{IV}_{1} (\mathbf{R}^{2})$ corresponding to the circle of radius $1 / i$, centred at the origin, with weight $i$, and a non-negative class 2 function $\phi$ with compact support, such that $\phi (x) = |x|^{2}$, on some open neighbourhood of the origin.
 \end{remark}

\begin{theorem}\label{Brakke flow:thm:1}
    Suppose $t_1,t_2\in \mathbf R$, $t_1<t_2$,
    $J = \{ s : t_1 \leq s \leq t_2\}$, $\mu$ is a function from $J$ to the set of Radon measures over $U$, and $\mu$ is a Brakke flow.
     Then the following nine statements hold.
    \begin{enumerate}[label=\ref{Brakke flow:thm:1}.\arabic{enumi}]
    \item \label{Brakke flow:thm:1:1} Whenever $t_1 \leq s_1\leq s_2\leq s_1+R^2/(2k)$, $\mathbf B(x_0,2R)\subset U$,
        \begin{equation}
            \mu(s_2)(\mathbf U(x_0,R)) \leq 16 \mu(s_1)(\mathbf U(x_0,2R))
            \label{spherical barrier estimate}
        \end{equation} 
    \item \label{Brakke flow:thm:1:2}
    Whenever $K$ is a compact subset of $U$, there exists $0\leq M<\infty$ such that 
    \begin{gather}
        \sup \{ \mu(t)(K) :t_1 \leq t \leq t_2\}\leq M \label{Brakke flow mass bound},\\
        \sup \{ \mathscr B(\mu(t),\phi) :t_1\leq t \leq t_2\} \leq M\cdot \boldsymbol{\nu}_K^2(\phi) \label{Brakke flow variation bound}
    \end{gather}
    for every nonnegative $\phi : U \to \mathbf R$ of class $2$ with compact support in $K$.
    \item 
    \label{Brakke flow:thm:1:3}
    Whenever $t_1 < a \leq t_2$ ($t_1 \leq a < t_2$),
    \begin{equation}
        \mu(a) \leq \lim\limits_{t\uparrow a} \mu(t), \ (resp. \ \mu(a) \geq \lim\limits_{t\downarrow a} \mu(t)).
    \end{equation}
    \item 
    \label{Brakke flow:thm:1:3-add}
    $\mu$ is continuous except at an at most countable subset of $J$.
    \item 
    \label{Brakke flow:thm:1:3-add-add}
    Whenever $t_1 < t \leq t_2$ ( \ $t_1 \leq t < t_2$\ ) , $0\leq \phi \in \mathscr K(U)$, and
    \[ -\infty < \D_s^- \mu(s)(\phi)|_t, \ (-\infty < \D_s^+ \mu(s)(\phi)|_t ),\]
    the following equality holds
    \[ \mu(t) (\psi) = \lim\limits_{s\uparrow t} \mu(s)(\psi), \ (resp. \ \mu(t) (\psi) = \lim\limits_{s\downarrow t} \mu(s)(\psi) )\]
    for $\psi \in \mathscr K(U)$ with 
    $\{ x : \psi(x)\neq 0 \} \subset \{ x : \phi(x)>0\}$.
    \item
    \label{Brakke flow:thm:1:4}
    For $\mathscr L^1$ almost every $t$ with 
    $t_1\leq t \leq t_2$, there exists a $V(t)\in \IVar_k(U)$ such that 
    \begin{equation}\label{V is regular at time t}
     \|V(t)\| = \mu(t), \ \bfh(V(t),\cdot) \in \mathbf L_2^{loc}(\mu(t),\mathbf R^n).
    \end{equation} 
    \item 
    \label{Brakke flow:thm:1:5}
    Whenever $\phi : U \to \mathbf R$
    is nonnegative of class $2$ with compact support, 
    $t \mapsto \mathscr B(\|V(t)\|,\phi)$ is $\mathscr L^1\weight J$ measurable.
    \item
    \label{Brakke flow:thm:1:6}
    The function $t\mapsto V(t)$ is $\mathscr L^1\weight J$ measurable.
    \item 
    \label{Brakke flow:thm:1:7}
    For every  $\phi :  U \to \mathbf R$ nonnegative of class $2$
    with compact support,
    \begin{gather}
        \begin{aligned}
        \mu(b)(\phi) -\mu(a)(\phi)  
        \leq \textstyle \int_a^b \mathscr B(\mu(t),\phi) \ d \mathscr L^1_t\\
        \text{for $a,b\in J$ with $a\leq b$}.
        \end{aligned}
        \label{Brakke flow:thm:1:7:2}
    \end{gather} 
    \end{enumerate}
\end{theorem}

\begin{proof}
    Proof of \ref{Brakke flow:thm:1:1}, 
     the proof is adapted from \cite[3.3]{Lahiri2017EqualityOT}, and we include it here for the reader's convenience.
     For every $(t,x)\in \mathbf R \times \mathbf R^n$, we define 
    \[ \varphi(t,x) = \sup \{ 1 -(2R)^{-2}(|x-x_0|^2 +2k(t-s_1)),0\}.\]
    One readily verifies that for every $t \in \mathbf R$,
        $\varphi(t,\cdot)\leq 1$;
    if $s_1 \leq t  <\infty$, 
    then 
        $\{ x :\varphi(t,x)>0\} \subset \mathbf U(x_0,2R)$;
    if $s_1\leq t \leq s_1 + \frac{R^2}{2k}$ and $|x-x_0|\leq R$, 
    then 
        $\varphi(t,x)\geq 1/2$.
    We compute for every  $p \in \mathscr P$ with $p\geq 3$, $x\in U$ and $s\in \mathbf R$, 
        \begin{gather*}
            \grad \varphi^p(s,x) = - p \cdot \varphi^{p-1}(s,x) \cdot (2R)^{-2} \cdot 2 (x-x_0),\\
            |S[\grad \varphi^p(s,x)]|^2 = \frac{p^2}{4R^4} \varphi^{2p-2}(s,x) |S(x-x_0)|^2 \label{thm:Brakke flow:1:1:1},\\
            S \bullet \D \grad \varphi^{p}(s,x) = \frac{-kp}{2R^2} \varphi^{p-1}(s,x) + \frac{p(p-1)}{4R^4} \varphi^{p-2}(s,x)|S(x-x_0)|^2\label{thm:Brakke flow:1:1:2}.
        \end{gather*}
    By \ref{Elementary estimate:from diff eq},
    for $s_1 \leq t \leq s_2$,  
        \begin{equation}\label{eq:spherical barrier}
            \textstyle \mathscr B(\mu(t),\varphi^4(t,\cdot))\leq \int \frac{4k}{2R^2}\varphi^{3}(t,x) \ d \mu(t)_x.
        \end{equation}
        
        Next, we will show that for $s_1 < t \leq t_2$,
         \[\D^-  \mu(t)(\varphi^4(t,\cdot)) \leq  \mathscr B(\mu(t),\varphi^4(t,\cdot)) -  2k R^{-2} \mu(t)(\varphi^3(t,\cdot)) \leq 0. \]
    By the mean value theorem, and nonincreasing property of $\varphi^3(s,\cdot)$ in variable $s$, 
    \[ \varphi^4(t-\delta,x) - \varphi^4(t,x) \geq  2k \delta R^{-2} (\varphi^3(t,x)) \]
    for $(t,x)\in \mathbf R \times \mathbf R^n$ and $0<\delta <\infty$.
    Since  $\D^-\mu(s)(\varphi^3(t,\cdot))|_t <\infty$, we have 
    \[ \liminf\limits_{\delta \downarrow 0} \mu(t-\delta)(\varphi^3(t,\cdot)) \geq \mu(t)(\varphi^3(t,\cdot)),\]
    the assertion then follows from expressing the finite difference 
    of $\D^- \mu(s)(\varphi^4(s,\cdot))|_t$ and \eqref{eq:spherical barrier}.

    By \cite[A.1]{Lahiri2017EqualityOT} with $a=s_1,b=s_2$, $f(s) = \mu(s)(\varphi^4(s,\cdot ))$ for $a \leq s \leq b$, and $L = 0$, and note 
    that $s_2 \leq s_1 + R^2/2k$, we have
        \[ (16)^{-1}\mu(s_2)(\mathbf U(x_0,R)) \leq \mu(s_2)(\varphi^4(s_2,\cdot)) \leq \mu(s_1)(\varphi^4(s_1,\cdot)) \leq \mu(s_1)(\mathbf U(x_0,2R)).\]
        
    Proof of \ref{Brakke flow:thm:1:2},
   the proof is adapted from \cite[3.6]{Lahiri2017EqualityOT}.
    Note that \eqref{Brakke flow variation bound} follows from 
    \eqref{Brakke flow mass bound} and 
    \cite[6.6]{Ilmanen-EllipticRegularization}, 
    and thus it is sufficient to prove 
    \eqref{Brakke flow mass bound}.
   Let $0< R = \sup \{ \dist(x, \mathbf R^n \sim U) : x \in K\}/4<\infty$. Then 
    \[ \mathbf B(x,2R) \subset U \ \text{for $x\in K$}.\]
    There exists finite set $B \subset K$ such that $K \subset \bigcup\{ \mathbf U(x,R) : x\in B\}$, and 
    there exists $t_1\leq u_1,\dots,u_N\leq t_2$ such that 
\[ \{s : t_1 \leq t \leq t_2 \} \subset \bigcup_{i=1}^N \{ s : u_i \leq s \leq u_i + (2 k)^{-1} R^{2} \}. \]
     We define the constant $M$ to be 
    \[  \sup \left\{ 16 \cdot \sum_{x\in B} \mu(u_i)(\mathbf U(x,2R)) : i \in \{ 1,\dots,N\}\right\},\]
    whenever $i\in \{1,\dots,N\}$ and $u_i \leq t \leq u_i + \frac{R^2}{2k}$, 
    we estimate  using \ref{Brakke flow:thm:1:1} with 
    $s_1 = u_i$, $s_2  = t$, and $x_0 = x \in B$ to obtain 
    \[ \mu(t)(K) \leq \sum_{x\in B} \mu(t)(\mathbf U(x,R)) \leq 16 \sum_{x\in B} \mu(u_i)(\mathbf U(x,2R)).\]
    Proof of \ref{Brakke flow:thm:1:3} and 
    \ref{Brakke flow:thm:1:3-add},
    the proof is identical to \cite[4.2]{Lahiri2017EqualityOT},
    by \ref{Brakke flow:thm:1:2}, there exists a positive number such that 
    \[ \sup \{ \overline{\D}_s\mu(s)(\phi)|_t:t_1 \leq t \leq t_2\}\leq M \cdot \boldsymbol{\nu}_K^2(\phi)\]
    for every nonnegative $\phi : U \to \mathbf R$ of class $2$ with  support in a compact subset $K$ of $U$,
    the assertions then follows from \cite[A.1]{Lahiri2017EqualityOT}.

Proof of \ref{Brakke flow:thm:1:3-add-add}, the proof is adapted from \cite[4.3]{Lahiri2017EqualityOT}, let
\begin{gather*}
    L(f) = \lim\limits_{s\uparrow t}\mu(s)(f)- \mu(t)(f) \quad (\mu(t) -\lim\limits_{s\downarrow t}\mu(s)(f) \quad resp.)\quad \text{for $f \in \mathscr K(U)$},
\end{gather*}
by \ref{Brakke flow:thm:1:3} and the Riesz representation theorem,
$L$ is represented by a Radon measure $\nu$ over $U$.
The conclusion then follows as $-\infty < \D^-\mu(s)(\phi)|_t$ ($-\infty < \D^+\mu(s)(\phi)|_t$ resp.) implies that $\nu(\{ x : \phi(x)>0\}) = 0$.

Proof of \ref{Brakke flow:thm:1:4},
choosing a sequence $\phi_i$ of compactly supported smooth functions on $U$ with 
        $\spt \phi_i \subset \Int \spt \phi_{i+1}$ and $U =  \bigcup_{i} \spt \phi_i $.
        For every $b,a\in J$ with $a\leq b$, we apply 
    \ref{Brakke flow:thm:1:2} with $K = \spt \phi_i$ to obtain $M_i$, and letting
    $L_i = M_i \cdot \nu^2_{K}(\phi_i)$, we have 
    \[ \mathscr B(\mu(t),\phi_i) \leq L_i \ \text{for $t_1\leq t \leq t_2$}.\]
   
    Let $A_i$ be the set of $t\in J$ such that 
    $-\infty<\mathscr B(\mu(t),\phi_i)$.
   By
    \cite[A.2]{Lahiri2017EqualityOT}, $\mathscr L^1(J\sim \bigcap_{i=1}^\infty A_i) = 0$, 
    one readily verifies that \eqref{V is regular at time t} holds for $t\in \bigcap_{i=0}^\infty A_i$. 
    
    Proof of \ref{Brakke flow:thm:1:5}, let 
    $A$ be the set of $t\in J$ such that 
     $\mu$ is continuous at $t$ and $V(t)\in \IVar_k(U)$,
      by \ref{Brakke flow:thm:1:4} and \ref{Brakke flow:thm:1:3-add}, $\mathscr L^1(J\sim A) = 0$.
    By \ref{Brakke flow:thm:1:2} and \ref{prop:interior convergence}, 
    $\mathscr B(\mu(\cdot),\phi)$ is upper semicontinuous on $A$,
the conclusion then follows.

    Proof of \ref{Brakke flow:thm:1:6},
    for every $\alpha \in \mathscr K(\G_k(U))$ and $r\in \mathbf R$, choose a nonnegative $\phi : U \to \mathbf R$ class $2$ with compact support such that 
   \[ \Clos \{ x : \alpha(x,\cdot)\neq 0\} \subset \{ x : \phi(x)>0\}, \]
    let $A$ be the set of $t\in J$ such that $\mu$ is continuous at $t$ and 
    $-\infty < \mathscr B(\|V(t)\|,\phi)$,
   and for every $i \in \mathscr P$, let
   \[ A(i) = A \cap \{ t : \mathscr B(\|V(t)\|,\phi)\geq -i \}.\]
   By \ref{prop:interior convergence}, 
   $\{t : V(t)(\alpha)\geq r\} \cap A(i)$
   is relatively closed in $A$.
   The conclusion then follows from
   \ref{Brakke flow:thm:1:4} and 
   \ref{Brakke flow:thm:1:5}.
   
   Proof of \ref{Brakke flow:thm:1:7}, 
  \eqref{Brakke flow:thm:1:7:2} follows from \ref{Brakke flow:thm:1:2} and \cite[A.2]{Lahiri2017EqualityOT}.
\end{proof}

\begin{definition}\label{definition of Brakke flow}
    Suppose 
    $J$ is an interval, 
    $U$ is an open subset of $\mathbf R^n$. 
    We call a function $V : J \to \Var_k(U)$ a Brakke flow if $\mu(t) = \|V(t)\|$, for $t\in J$, is a Brakke flow, and for $\mathscr L^1$ almost every $t\in J$, $V(t)\in \IVar_k(U)$.
\end{definition}

\section{Brakke flow: space time measure}

\begin{theorem}\label{thm:Brakke flow:space time measure}
    Suppose $J$ is an open interval and 
    $U$ is an open subset of $\mathbf R^n$,
    $V: J \to \Var_k(U)$ is a Brakke flow.
    Then the following four statements hold 
    \begin{enumerate}[ref=\ref{thm:Brakke flow:space time measure}.\arabic{enumi}]
        \item \label{Brakke flow:space time measure:1}
        There exists a Radon measure $\lambda$ over $J \times \G_k(U)$ characterised by 
        \[ \textstyle \lambda(f) = \int f(t,x,S) \ d V(t)_{(x,S)} \ d \mathscr L^1_t \]
        for $f\in \mathscr K(J\times \G_k(U))$.
        \item \label{Brakke flow:space time measure:2}Whenever $g : J \times \G_k(U) \to \overline{\mathbf R}$ is a $\lambda$ integrable function, $N$ is the set of $t$ such that $g(t,\cdot,\cdot)$ is not $V(t)$ integrable, we have $\mathscr L^1(N) = 0$ and 
        \[\textstyle  \int g \ d \lambda = \int \int g(t,x,S) \ d V(t)_{(x,S)} \ d \mathscr L^1_t.\]
        \item \label{Brakke flow:space time measure:3}
        If $N \subset J$ $\mathscr L^1(N) = 0$, then $\lambda(N\times \G_k(U)) = 0$.
        \item \label{Brakke flow:space time measure:4}
        If $I \subset  J$ is $\mathscr L^1$ measurable and $B \subset U \times \G(n,k)$ is a Borel set, then $I\times B$ is $\lambda$ measurable.
    \end{enumerate}
\end{theorem}

\begin{proof}
    Proof of \ref{Brakke flow:space time measure:1},
    noting by \ref{Brakke flow:thm:1:2}, for every $0\leq f \in \mathscr K(U)$ and $a,b\in J$ with $a\leq b$,
    \[ \sup \{ \|V(t)\|(f) : a \leq t \leq b\} <\infty ,\]
    and by \ref{Brakke flow:thm:1:6} $t\mapsto V(t)(\alpha) \cdot \omega(t)$ is $\mathscr L^1\weight J$ measurable for $\omega \in \mathscr K(J)$ and 
    $\alpha \in \mathscr K(\G_k(U))$,
    and 
    $\{ \omega \cdot \alpha : \omega \in \mathscr K(J), \alpha \in \mathscr K(\G_k(U))\}$ is dense in $\mathscr K(J \times \G_k(U))$.
    We may then define the Radon measure $\lambda$ associated with the following linear functional 
    \[ \textstyle \lambda(f) = \int_J \int f(t,x,S) \ d V(t)_{(x,S)} \ d \mathscr L^1_t \ \text{for $f\in \mathscr K(J\times \G_k(U))$}.\]
    Proof of \ref{Brakke flow:space time measure:2},
    let $F$ be the class of all  $\lambda$ measurable $g: J \times \G_k(U) \to \overline{\mathbf R}$ which satisfies the conclusions of \ref{Brakke flow:space time measure:2}. Whenever
    \begin{gather*}
        g_1,g_2,\dots \in F, \\
        g(t,x,S) = \lim\limits_{i\to\infty} g_i(t,x,S) \ \text{for $(t,x,S) \in J\times \G_k(U)$},
    \end{gather*}
    let $N_g$ be the set of $t\in J$ such that $g(t,\cdot,\cdot)$ is not $V(t)$ integrable.
    If $0\leq g_1,g_2,\uparrow g$,
then for every $t\in J \sim \bigcup_{i=1}^\infty N_{g_i}$, by the Lebesgue monotone convergence theorem
\[ \textstyle \int g(t,x,S) \ d V(t)_{(x,S)} = 
\lim\limits_{i\to\infty}\int g_i(t,x,S) \ d V(t)_{(x,S)},
\]
 and we infer that $g$ belongs to $F$.

 Suppose $f :J \times \G_k(U) \to \overline{\mathbf R}$ is a $\lambda$ summable function belonging to $F$,
$g_i\in F$ for $i\in \mathscr P$, and 
\begin{gather*}
    |g_i(t,x,S)|\leq f(t,x,S)  \ 
    \text{for $(t,x,S) \in J\times \G_k(U)$} , \\
     g(t,x,S) = \lim\limits_{i\to\infty}g_i(t,x,S) \ 
    \text{for $(t,x,S) \in J\times \G_k(U)$}.
\end{gather*}
Note that for $\mathscr L^1$ almost every $t\in J$, 
\[\textstyle  \int |g_i(t,x,S)| \ d V(t)_{(x,S)}\leq \int f(t,x,S) \ d V(t)_{(x,S)} <\infty,\] 
by the Lebesgue dominated convergence theorem, we have
\[ \textstyle \int g(t,x,S) \ d V(t)_{(x,S)} = \lim\limits_{i\to\infty} \int g_i(t,x,S) \ d V(t)_{(x,S)},\]
 and we infer, again from the Lebesgue dominated convergence theorem that
\[ \textstyle  \int\int g(t,x,S) \ d V(t)_{(x,S)} \ d \mathscr L^1_t = \lim\limits_{i\to\infty} \int\int g_i(t,x,S) \ d V(t)_{(x,S)}, \]
therefore $g$ belongs to $F$.

By splitting $g$ into $g^+,g^-$ and approximating $g^+,g^-$ by $\lambda$ measurable function with countable image using \cite[2.3.3]{MR41:1976}, we reduced the problem to $g$ equals a characteristic function $\varphi$ of $\lambda$ measurable set $A$ with finite $\lambda$ measure. Now, the assertion follows from noting that $\varphi\in F$ if $A$ is an open set, because the topology admits a countable base.

Proof of \ref{Brakke flow:space time measure:3},
whenever $K \subset \G_k(U)$ is a compact set and 
$I_1,I_2,\dots$ are decreasing open sets containing $N$ such that $\lim\limits_{i\to\infty }\mathscr L^1(I_i) = 0$, one readily verifies by \ref{Brakke flow:space time measure:2} and the Lebesgue dominated convergence theorem that 
\[ \lambda(N\times K)\leq \lim\limits_{i\to\infty} \lambda(I_i\times K) =   0,\]
and the conclusion follows from expressing $\G_k(U)$ as a countable union of compact sets.

Proof of \ref{Brakke flow:space time measure:4}, 
covering $J$ by a countable family of $\mathscr L^1$ measurable subsets of $J$, we may assume $I$ has finite $\mathscr L^1$ measure, the conclusion then follows by approximating $I$ by its compact subsets and \ref{Brakke flow:space time measure:3}.
\end{proof}

\begin{lemma}\label{Borel regularity lemma:1}
    Suppose $n, \nu\in \mathscr P$.
    Then the function 
    \[ (T,\theta) \mapsto T(\theta), \ \text{for $(T,\theta)\in \mathscr K'(\mathbf R^n, \mathbf{R}^{\nu})\times \mathscr K(\mathbf R^n, \mathbf{R}^{\nu})$}\]
    is Borel.
\end{lemma}

\begin{proof}
    We denote the function $(T,\theta)\mapsto T(\theta)$ by $F$.
    For every $i,j\in \mathscr P$, one readily verifies
    $\mathscr K'(\mathbf R^n, \mathbf{R}^{\nu})\times\mathscr K(\mathbf R^n, \mathbf{R}^{\nu}) = \bigcup_{i,j\in \mathscr P}  A(i,j) \times \mathscr K_{\mathbf B(0,i)}(\mathbf R^n, \mathbf{R}^{\nu})$
    where 
    \[ A(i,j) = \mathscr K'(\mathbf R^n)\cap \{ T : \|T\| (\mathbf B(0,i))\leq j\}.\]
   One readily verifies that $F|A(i,j)\times \mathscr K_{\mathbf B(0,i)}(\mathbf R^n, \mathbf{R}^{\nu})$ is continuous.
\end{proof}
\begin{definition}
    Whenever $X$ is a locally compact space admitting a countable base,
    $\mathscr M_+(X)$ denotes the subspace of $\mathscr K'(X)$ consisting of Radon measures.
\end{definition}

\begin{remark}\label{lem:Borel graph lemma}
    Suppose $X$ and $Y$ are complete separable metric spaces, $Z_{1}, \ldots, Z_{m}$ are separable, locally compact metric spaces for $m \in \mathscr{P}$, $G$ is a Borel subset of $(\prod_{i = 1}^{m} \mathscr K'(Z_{i},\mathbf R^{\nu_{i}}) ) \times X \times  Y$, with $\nu_{1}, \ldots, \nu_{m} \in \mathscr{P}$,
    \begin{align*}
        & \textstyle p(T,x,y) = (T,x), \quad  \text{for $(T,x,y)\in (\prod_{i = 1}^{m} \mathscr K'(Z_{i},\mathbf R^{\nu_{i}}) ) \times X \times  Y$}
    \end{align*}
    and $p|G$ is univalent. 
    Then the function defined by $g(T,x)=y$, whenever $(T,x,y)\in G$, is a Borel function and $\dmn g = p(G)$ is a Borel subset of $(\prod_{i = 1}^{m} \mathscr K'(Z_{i},\mathbf R^{\nu_{i}}) ) \times X$.
\end{remark}

\begin{proof}
    We prove the statement for $m = 1$, with the general statement following by an identical argument.
    For every $(T,x,y)\in \mathscr K'(Z,\mathbf R^{\nu}) \times X  \times Y$,
    we define $q(T,x,y) = y$.
    and observe that 
    if $U$ is an open subset of $Y$, 
    then $g^{-1}(U) = p(q^{-1}(U) \cap G)$.
    The conclusion  then follows from showing that
    $p(B)$ is Borel, for any Borel subset $B \subset G$.
    By \cite[2.1, 2.23]{Menne-WDFV}, there exists 
    a continuous bijection $f$ from a complete, separable metric space $W$, to $\mathscr K'(Z,\mathbf R^\nu)\times X$.
    We denote $F (w, y) = (f (w), y)$, for $w \in W$, and $y \in Y$.
    Since $f$ is a continuous bijection, we have that $F^{-1}(B)$ is a Borel set, and $(p \circ F) |F^{-1} (B)$ is univalent, and thus
    the assertion follows from  \cite[2.2.10, p. 67]{MR41:1976}.
\end{proof}

\begin{lemma}\label{general measurability of quotient}
    Suppose $n \in \mathscr P$ and
    $b_{a,r}$ denotes   the characteristic function of $\mathbf B(a,r)$. Then the set $\bfh$ of $(T,\phi,a,\alpha) \in \mathscr K'(\mathbf R^n,\mathbf R^\nu)\times \mathscr M_+(\mathbf R^n)\times \mathbf R^n\times \mathbf R^\nu$ such that 
    \[  v\bullet \alpha = \lim\limits_{r\downarrow 0 }\frac{T(b_{a,r}\cdot v)}{\phi(\mathbf B(a,r))} \ \text{for $v\in \mathbf R^\nu$}\]
    is a Borel set, where $b_{a,r}$ is the characteristic function of $\mathbf B(a,r)$.
    Moreover, the function mapping $(T,\phi,a)\in \mathscr K'(\mathbf R^n,\mathbf R^\nu) \times \mathscr M_+(\mathbf R^n)\times \mathbf R^n$ onto $\alpha \in \mathbf R^\nu$ such that $(T,\phi,a,\alpha)\in \bfh$ is a Borel function with Borel domain. 
\end{lemma}
\begin{proof}
    For each $0<r<\infty$
    and $v\in \mathbf R^n$,
    we define
    \begin{gather}
        G_r^v(T,\phi,a) = T(b_{a,r}\cdot v)\cdot \phi(\mathbf B(a,r))^{-1}  \ \text{if $0<\phi(\mathbf B(a,r))<\infty$}\label{meas func:1}\\
        M_r(T,\phi,a) = \phi (\mathbf B(a,r)), \label{meas func:2}\\
        F_r^v(T,\phi,a) = T(b_{a,r}\cdot v), \label{meas func:3}
    \end{gather}
    \text{for $T\in \mathscr K'(\mathbf R^n,\mathbf R^\nu)$, $a\in \mathbf R^n$, and $\phi \in \mathscr M_+(\mathbf R^n)$}.
    For every $0<\kappa,\gamma<\infty$, we define $h_\gamma^\kappa$ to be the set of 
    $(T,a,\phi,\xi)$ such that 
    \[\textstyle \sum_{i = 1}^{\nu} |e_{i}\bullet \xi - G^{e_{i}}_r(T,\phi,a)|\leq \kappa \ \text{for $0<r<\gamma$},\]
    where $e_{1}, \ldots, e_{\nu}$ are the unit coordinate vectors on $\mathbf{R}^{\nu}$. 
    
    We will show that 
    the functions defined in \eqref{meas func:1},
    \eqref{meas func:2}, and \eqref{meas func:3}
    are Borel functions with Borel domain, then using the right continuity of $G^v_r(T,\phi,a)$ in variable $r$, one readily verifies each $h^\kappa_\gamma$ is Borel, and $\bfh = \bigcap_{i=1}^\infty \bigcup_{j=1}^\infty h_{j^{-1}}^{i^{-1}}$.
    It is sufficient to prove that \eqref{meas func:2}
    and \eqref{meas func:3} are Borel, \eqref{meas func:2} is upper semicontinuous, that is if $(a_i,\mu_i) \to (a,\mu)$ as $i\to \infty$, then 
    \[ \limsup\limits_{i\to\infty} \mu_i(\mathbf B(a_i,r)) \leq \mu(\mathbf B(a,r)),\]
    in fact, whenever
    $U$ is an open set containing $\mathbf B(a,r)$, we may choose a nonnegative $f\in \mathscr K(\mathbf R^n)$ so that $0\leq f \leq 1$, $f(x)=1$ for $x\in \mathbf B(a_i,r)$ 
    for $i$ large enough, and $\mathbf B(a_i,r) \subset U$,
    we estimate
    \[ \limsup\limits_{i\to\infty} \mu_i(\mathbf B(a_i,r))\leq \lim\limits_{i\to\infty}\mu_i(f) = \mu(f) \leq \mu(U).\]
    To show the function in \eqref{meas func:3} is Borel, we choose $\eta \in \mathscr D(\mathbf R)$ with $\eta(t)=1$ for $t\leq 0$, 
    $\eta(t)=0$ for $t\geq 1$, and 
    $\eta'(t)\leq 0$ for $t\in \mathbf R$. Whenever $0<\varepsilon<\infty$, we define 
    $\varphi_{a,r,\varepsilon}\in \mathscr K(\mathbf R^n)$ by the following equation,
    \begin{equation*}
        \varphi_{a,r,\varepsilon}(x) = \eta\left(\frac{|x-a|-r}{\varepsilon}\right) \ \text{for $x\in \mathbf R^n$},
    \end{equation*}
    and one readily verifies 
    \begin{gather*}
    0\leq \varphi_{a,r,\varepsilon} \leq 1, \ 
    \spt \varphi_{a,r,\varepsilon} \subset \mathbf B(a,r+\varepsilon)  ,\\
        \lim\limits_{\varepsilon\downarrow 0} \varphi_{a,r,\varepsilon} = b_{a,r},\\
        \lim\limits_{\varepsilon \downarrow 0} T(\varphi_{a,r,\varepsilon} \cdot v) = T(b_{a,r} \cdot v)
        \ \text{for $T\in \mathscr K'(\mathbf R^n, \mathbf{R}^{\nu})$},
    \end{gather*} 
    and the function mapping $a\in \mathbf R^n$ onto $\theta(\cdot-a)\in \mathscr K(\mathbf R^n)$ is continuous.
    The assertion then follows from \ref{Borel regularity lemma:1}.
   The postscript follows from \ref{lem:Borel graph lemma}.
 \end{proof}

\begin{theorem}\label{thm:measurable function:tangent,density,mean curvature vector}
    Suppose $U$ is an open subset of $\mathbf R^n$. We define the following three relations
    \begin{enumerate}[label=\ref{thm:measurable function:tangent,density,mean curvature vector}.\arabic{enumi}]
        \item \label{thm:measurable function:tangent,density,mean curvature vector:1} $\bfT$ is the set of $(V,a,T) \in \Var_k(U) \times U \times \G(n,k)$ such that $V\in \Var_k(U)$, $a\in U$, and $T\in \G(n,k)$ such that $T$ is the unique tangent plane of $V$ at $a$.
        \item \label{thm:measurable function:tangent,density,mean curvature vector:2}
        $\bfh$ is the set of $(V,a,\xi)\in \Var_k(U) \times U \times \mathbf R^n$ such that  $\|\delta V\|$ is a Radon measure and
        \[ v \bullet \xi = \lim\limits_{r\downarrow 0}\frac{\delta V(b_{a,r}\cdot v)}{\|V\|(\mathbf B (a,r))} \ \text{for $v\in \mathbf R^n$}\]
        where $b_{a,r}$ is the characteristic function of $\mathbf B(a,r)$.
        \item 
        \label{thm:measurable function:tangent,density,mean curvature vector:3}
        $\Density$ is the set of $(V,a,\theta) \in \mathbf{V}_{k} (U) \times U \times \mathbf R$ such that $\Density^k(\|V\|,a)= \theta$.
    \end{enumerate}
     Then $\bfT,\bfh,\Density$ are Borel sets. Moreover, whenever 
     $\lambda$ and $V$ are as in \ref{thm:Brakke flow:space time measure}, the function mapping $(t,x)$ onto $(T,\xi,\theta) \in \G(n,k) \times \mathbf R^n \times \mathbf R$ such  that 
     \[ (V(t),x,T)\in \bfT, \ (V(t),x,\xi) \in \bfh, \ (V(t),x,\theta)\in \Density \]
     is a $\|\lambda\|$  measurable function with $\|\lambda\|$  measurable domain.
\end{theorem}

\begin{proof}
    For each $a\in U$, choose $0<\delta <\infty$ and $\zeta \in \mathscr D(U)$ with $\mathbf B(a,\delta) \subset U$,
    $0\leq \zeta \leq 1$ and $\zeta(x)=1$ for $x\in \mathbf B(a,\delta)$, we define 
    \[ L_\zeta(V,a) = (V_\zeta,a) \in \Var_k(\mathbf R^n)\times \mathbf R^n \ \text{for $V\in \Var_k(U)$ and $a\in U$}\]
    where $V_\zeta$ is the extension of $V\weight \zeta$ by zero to $\mathbf R^n$.
    One readily verifies 
    \begin{gather*}
        \bfT(V,x) = \bfT'\circ L_\zeta(V,x),\\
        \bfh(V,x) = \bfh' \circ L_\zeta(V,x),\\
        \textstyle \Density(V,x) = \Density'\circ L_\zeta(V,x)
    \end{gather*}
    for every $V\in \Var_k(U)$ and $x\in \mathbf B(a,\delta)$,
    where $\bfT'$, $\bfh'$, and $\Density'$ are the set as in \ref{thm:measurable function:tangent,density,mean curvature vector}
    with $U = \mathbf R^n$.
    In view of \ref{lem:Borel graph lemma}, it is sufficient to prove the case $U = \mathbf R^n$.

    For every $0<\varepsilon,\delta <\infty$, and $\alpha \in \mathscr{K}(\G_k(\mathbf{R}^{n}))$, let
    $B(\varepsilon,\delta,\alpha)$ be the  set  of $ (V,a,W) \in \Var_k(\mathbf R^n) \times \mathbf R^n \times \Var_k(\mathbf R^n)$
    such that 
    \[ \textstyle |W(\alpha) - r^{-k}\int \alpha(r^{-1}(x-a),S) \ d V_{(x,S)}|\leq \varepsilon \ \text{for $0<r\leq \delta$},\]
    clearly $B(\varepsilon,\delta,\alpha)$ is closed,
    and let $B$ be the set of $(V,a,W)$ such that 
    \[ \textstyle W(\alpha) = \lim\limits_{r\downarrow 0} r^{-k}\int \alpha(r^{-1}(x-a),S) \ d V_{(x,S)} \
     \text{for $\alpha \in \mathscr K(\G_k(\mathbf R^n))$},\]
     because
     $B = \bigcap_{\alpha \in D}\bigcap_{i=1}^\infty \bigcup_{j=1}^\infty B(i^{-1},j^{-1},\alpha)$
     for some countable dense subset $D$ of $\mathscr K(\G_k(\mathbf R^n))$,
     $B$ is a Borel set.
     
     Proof of \ref{thm:measurable function:tangent,density,mean curvature vector:2},
     as the function $(V,a) \mapsto (\delta V,\|V\|,a)$
      is continuous, the assertion then follows from
      \ref{general measurability of quotient}.

    Proof of \ref{thm:measurable function:tangent,density,mean curvature vector:3}, whenever $i,j\in \mathscr P$, let 
    $\Density(i,j)$ to be the set of $(\phi,a,\theta)$ such that 
    \[ \left|\theta - \frac{\phi(\mathbf B(a,r))}{\boldsymbol{\alpha} (k) r^k}\right| \leq i^{-1}  \ \text{for $0<r<j^{-1}$}.\]
    Note that $\Density(i,j) = \bigcap_{0<r<j^{-1}} \Density(i,j,r)$, where 
    \[ \Density(i,j,r) = \mathscr M_+(\mathbf R^n)\times \mathbf R^n \times \mathbf R \cap 
    \left\{(\phi,a,\theta) : \left| \theta -\frac{\phi(\mathbf B(a,r))}{\boldsymbol{\alpha} (k) r^k}\right| \leq i^{-1} \right\}\]
    is Borel for every $0<r<\infty$.

    From the right continuity of $\phi(\mathbf B(a,r))$ in variable $r$, we infer 
    \[ \Density(i,j) = \bigcap \{ \Density(i,j,r):0<r<j^{-1}, \ \text{$r$ is rational}\}\]
    is a Borel set. Then the conclusion follows as $\Density = \bigcap_{i=1}^\infty \bigcup_{j=1}^\infty \Density(i,j)$. 
    Moreover, as $(V, a) \mapsto (\| V \|, a)$ is continuous, by \ref{lem:Borel graph lemma} we have that the function $(V, a) \mapsto \Density^{k} (\| V \|, a)$ is a Borel function, with Borel domain.
    
    Proof of \ref{thm:measurable function:tangent,density,mean curvature vector:1}, we use 
    \ref{thm:measurable function:tangent,density,mean curvature vector:3} to define a Borel function by
    \[ \textstyle \pi(V,x,T) = (V,x,\Density^k(\|V\|,x)\cdot \vVarOp(T)),\]
    for $T\in  \G(n,k)$ and $(V,x,\Density^k(\|V\|,x))\in \Density$,
    and observe that $\bfT = \pi^{-1}(B)$ is a Borel set.
    
     To prove the postscript,
     we define 
    \[ F(t,x) = (V (t),x) \ \text{if $t\in \dmn V (\, \cdot \,)$ and $x\in U$},\]
    \[ \overline{\bfT}(V,x) = T \ \text{ if $(V,x,T) \in \bfT$},\]
    by \ref{Brakke flow:space time measure:4}, $F$ is $\|\lambda\|$ measurable.
    From \ref{lem:Borel graph lemma}, we infer that $\overline{\bfT}$ is a Borel function, with a Borel domain, therefore
    \[ \|\lambda\|(J\times U\cap \{ (t,x): (t,x) \notin \dmn \overline{\bfT}\circ F\}) =0,\]
    because for $\mathscr L^1$ almost every $t$, and $\| V (t) \|$ almost every $x$, $(t,x) \in \dmn \overline{\bfT}\circ F$, and thus by \ref{Brakke flow:space time measure:2},
    \[\textstyle \int \|V(t)\| (\{ x : (t,x) \notin\dmn \overline{\bfT}\circ F\}) \ d \mathscr L^1_t = 0.\]
    A similar argument holds for $\bfh$ and $\Density$ with $\bf T$ replaced by $\bfh$ and $\Density$.
\end{proof}

\begin{theorem}\label{def:Brakke flow}
    Suppose $J$ is an interval, $U$ is an open subset of $\mathbf R^n$, and
    $V : J \to \Var_k(U)$.
 Then the two statements are equivalent,
    \begin{enumerate}
        \item\label{def pt: time integrated definition of Brakke flow}
    For every 
    nonnegative $\phi : U \to \mathbf R$ of class $2$ with compact support,  and $a,b\in J$ with $a\leq b$,
    \begin{equation}\label{eq:Brakke intgeral ineq}
            \|V(b)\|(\phi) - \|V(a)\|(\phi)
            \leq \textstyle \int_a^b \mathscr B(\|V(t)\|,\phi) \ d \mathscr L^1.
    \end{equation}
    \item For every $t\in J$ and nonnegative $\phi : U \to \mathbf R$ of class $2$ with compact support,
    \begin{equation}\label{eq:Brakke differential ineq}
        \overline{\D}_s\|V(s)\|(\phi)|_t \leq \mathscr B(\|V(t)\|,\phi).
     \end{equation}      
    \end{enumerate}
\end{theorem}

\begin{proof}
    By \ref{Brakke flow:thm:1:7}, \eqref{eq:Brakke differential ineq} implies \eqref{eq:Brakke intgeral ineq}.
    It remains to show the converse, we assume 
    $-\infty < \D^+ \|V(t)\|(\phi)$, we can choose
    $t_i \downarrow t$ as $i\to\infty$ and  
    \[ \frac{\|V(t_i)\|(\phi)-\|V(t)\|(\phi)}{t_i-t} \leq  \mathscr B(\|V(s_i)\|,\phi) \]
    for some $t<s_i<t_i$, and by \ref{prop:interior convergence}, 
    passing to subsequence if necessary,
    we may also assume that 
    there exist $W \in \mathbf{V}_{k} (U)$ such that,
     \[ \limsup\limits_{i \to \infty}\mathscr B(\|V(s_i)\|, \phi) \leq \mathscr B(\|W\|,\phi)\]
     \[ W  = \lim\limits_{i\to\infty} V(s_i) \weight U_\phi.\]
     Arguing similarly to \cite[4.3]{Lahiri2017EqualityOT}, \ref{Brakke flow:thm:1:3-add-add} also holds in the setting of \ref{def:Brakke flow}.\ref{def pt: time integrated definition of Brakke flow}, implying that 
     \[ \| V (t) \|(\psi) =  \| W \| (\psi) \ \text{for $\psi \in \mathscr K(U)$ with $\spt \psi \subset U_\phi$},\]
     therefore $\|V(t)\|\weight U_\phi = \|W\|\weight U_\phi$.
     By a similar argument for $\D^-$, the conclusion follows.
     
\end{proof}

\begin{remark}
    The idea of the proof of \ref{def:Brakke flow} is taken from the proof of 
    \cite[Main result B5 implies B1]{Lahiri2017EqualityOT} (see also \cite[4.4]{AS-GeneralisedBrakkeFlow}).
\end{remark}
\begin{lemma}[[\protect{\cite[Theorem 3.28]{ambrosio2000functions}}]\label{lem:one dimensional BV}
    Suppose $J$ is an open interval, $u \in \mathscr D'(J)$, and $\|u'\|$ is a Radon measure, 
    \begin{enumerate}[label=\ref{lem:one dimensional BV}.\arabic{enumi}]
        \item \label{lem:one dimensional BV:1} If $\|u'\|(J)<\infty$, then there exists a left continuous  $u^l :J \to \mathbf R$, and a right continuous $u^r : J \to \mathbf R$ function, uniquely characterised by 
    \[ u(\omega) = \textstyle \int u^l(t)\cdot \omega(t) \ d \mathscr L^1_t \ \text{for $\omega \in \mathscr D(J)$},\]
     \[ u(\omega) = \textstyle \int u^r(t)\cdot \omega(t) \ d \mathscr L^1_t \ \text{for $\omega \in \mathscr D(J)$}.\]
     \item \label{lem:one dimensional BV:2} There exists an $\mathscr L^1$ locally summable function $v : J \to \mathbf R$ such that 
     \[ u(\omega) = \textstyle \int \omega(t)\cdot v(t) \ d \mathscr L^1_t \ \text{for $\omega \in \mathscr D(J)$}.\]
    \end{enumerate}
\end{lemma}

\begin{proof}
Proof of \ref{lem:one dimensional BV:1},
   the uniqueness property follows from the continuity, so 
   it is sufficient to establish the existence.
    In fact, one readily verifies by Fubini's theorem that the distributional derivative of 
    \[ w^l(t) = u'(J \cap \{ s : s <t\}) \ \text{for $t\in J$}\]
    \[ w^r(t) = u'(J \cap \{ s : s \leq t\}) \ \text{for $t\in J$}\]
    is $u'$. 
    By \cite[4.1.4]{MR41:1976} there exists a unique $c\in \mathbf R$ such that $u (t) = c + w^r (t)$ and $u (t) = c + w^l (t)$ for $\mathscr{L}^{1}$ almost all $t \in J$, and thus the existence follows by taking $u^l = c + w^l$
    and $u^r = c + w^r$.

    Proof of \ref{lem:one dimensional BV:2},
    for every relatively compact open subinterval $I$ of $J$, we denote $u_I$ to be the restriction of $u$ to $I$, one readily verifies that 
    $\|u_I'\|(I)<\infty$. 
    We apply \ref{lem:one dimensional BV:1} with $J$ and $u$ replaced by $I$ and $u_I$, to obtain a $u_I^l$ as in \ref{lem:one dimensional BV:1}. 
    Choosing $I_1,I_2,\dots,$ to be an increasing sequence of relatively compact open subintervals of $J$, covering $J$, and \ref{lem:one dimensional BV:2} follows.
\end{proof}

\begin{remark}\label{additivity of canoical representation}
    The set $F$ of $u\in \mathscr D'(J)$ satisfying the hypothesis of $\ref{lem:one dimensional BV}$ is a vector subspace. From the uniqueness part of the conclusion of \ref{lem:one dimensional BV},  if $c\in \mathbf R$ and $u,v\in F$, then 
    \[ (cu+v)^l = c\cdot u^l + v^l, \ (cu+v)^r = c\cdot u^r + v^r.\]
\end{remark}

 We now prove {\bf Theorem \ref{intro:thm:4}}.

\begin{proof}
 Let $\mathscr{K}^+(\mathbf{G}_{k} (U))$ be the set of nonnegative members of $ \mathscr K(\G_k(U))$.
    For $\alpha \in \mathscr{K}^{+} (\mathbf{G}_{k} (U))$, we define  a Radon measure $R_\alpha$ on $J$ characterised by
    \begin{equation*}
        R_{\alpha} (\omega) = \textstyle{\int} \omega (t) \alpha (x, S) \, d V_{(t, x, S)}, \quad  \text{for } \omega \in \mathscr{D} (J).
    \end{equation*}
    First, proving for an appropriate countable dense subset of $\mathscr{K}^+(\mathbf{G}_{k} (U)) $ (by \cite[2.8.17, 2.9.5]{MR41:1976}), and then extending to the whole of $\mathscr{K}^{+} (\mathbf{G}_{k} (U))$, we have that there exists a $P \subset J$, with $\mathscr{L}^{1} (J \sim P) = 0$, such that for all $t \in P$, and $\alpha \in \mathscr{K}^{+} (\mathbf{G}_{k} (U))$,
    \begin{eqnarray*}
         R_{\alpha}^{+} (t) &=& \lim_{\varepsilon \searrow 0}  \varepsilon^{-1} R_{\alpha} ({\{ s \, \colon \, t \leq s < t + \varepsilon\}}), \\
         R_{\alpha}^{-} (t) &=& \lim_{\varepsilon \searrow 0} \varepsilon^{-1} R_{\alpha} ({\{ s \, \colon \, t - \varepsilon < s \leq t \}}), \\
         R_{\alpha} (t) &=& \lim_{\varepsilon \searrow 0} (2\varepsilon)^{-1} R_{\alpha} ({\{ s \, \colon \, t - \varepsilon < s \leq t + \varepsilon\}}), 
    \end{eqnarray*}
    all exist, and are finite, non-negative real numbers.
    For all $t \in P$, we may then define Radon measures $\langle V^{+}, t \rangle$, $\langle V^{-}, t \rangle$, and $\langle V, t \rangle$, over $\mathbf{G}_{k} (U)$, characterised by,
    \begin{equation*}
        \langle V^{+}, t \rangle (\alpha) = R_{\alpha}^{+} (t), \, \langle V^{-}, t \rangle (\alpha) = R_{\alpha}^{-} (t), \, \langle V, t \rangle (\alpha) = R_{\alpha} (t),
    \end{equation*}
    for $\alpha \in \mathscr{K}^{+} (\mathbf{G}_{k} (U))$.
    Moreover, there exists an $N \subset P$, with $\mathscr{L}^{1} (P \sim N) = 0$, such that $\langle V^{+}, t \rangle$, $\langle V^{-}, t \rangle$, and $\langle V, t \rangle$ are all equal for $t \in N$.
    
    Proof of \ref{Space time Brakke flow:0}, for $\phi : U \to \mathbf R$ a nonnegative function of class $2$ with compact support, we define 
    \[\textstyle T_{\phi}(\omega) = - R_{\phi}'(\omega) + \int \omega(t)\cdot \mathscr B(\|\langle V,t\rangle \|,\phi) \ d \mathscr L^1_t \ \text{for $\omega \in \mathscr D(J)$},\]
    by \eqref{eq:Brakke's distributional inequality}, $T_{\phi} \geq 0$, and by the
    Riesz representation theorem, $T_{\phi}$ is representable by a Radon measure, therefore 
    $R_{\phi}'$ is  representable by integration, and
    by \ref{lem:one dimensional BV:2} we have that $R_{\phi}$ will be absolutely continuous with respect to $\mathscr{L}^{1} \weight J$, and $R_{\phi} (t) = \langle V, t \rangle (\phi)$ is locally of bounded variation on $J$, completing the first conclusion of \ref{Space time Brakke flow:0}.
    Then, for $\alpha \in \mathscr{K}^{+} (\mathbf{G}_{k} (U))$, one may construct a non-negative, class 2 function $\phi \colon U \rightarrow \mathbf{R}$, with compact support, such that 
    \begin{equation*}
        \sup \{ \alpha (x, S) \, \colon S \in \mathbf{G} (n, k) \} \leq \phi (x),
    \end{equation*}
    for all $x \in U$.
    Therefore $R_{\alpha} \leq R_{\phi}$, and thus will be absolutely continuous with respect to $\mathscr{L}^{1} \weight J$, and hence, by \cite[2.9.7]{MR41:1976}, the second conclusion of \ref{Space time Brakke flow:0} follows.
    
    Proof of \ref{Space time Brakke flow:1} and \ref{Space time Brakke flow:2}.
     Suppose $\phi : U \to \mathbf R$ is a nonnegative function of class $2$ with compact support.
     Let $g(t) = \mathscr B(\|\langle V,t\rangle\|,\phi)$ for $t\in J$.
    By a suitable approximation of the characteristic function of a compact nondegenerate interval of $J$, we infer that
\begin{equation}\label{space time flow:priori local mass bound}
        \textstyle \|\langle V,b\rangle\|(\phi) - \|\langle V,a\rangle\| (\phi) \leq \int_a^b g(t) \ d \mathscr L^1_t \ \text{for $a,b\in J\sim N$},
    \end{equation}
    Whenever $a,b\in \Clos J$, $a\leq b$, $I = \{ s : a < s <b\}$, 
    we define 
        \begin{equation}
        \label{eq:distribution of time variable}
            \textstyle v(\omega) = \int \omega (t)\cdot \phi(x) \ d \|V\|_{(t.x)} \ \text{for $\omega \in \mathscr D(I)$},
        \end{equation}
        \begin{equation}
        \label{eq:the differerbce distribution of flow and linear support flow}
            S(\omega) = v(\omega) - \textstyle \int L \cdot t\cdot  \omega(t) \ d \mathscr L^1_t \ \text{for $\omega \in \mathscr D(I)$},
        \end{equation}
        and assume that $0 \leq L < + \infty$ is such that
        \begin{equation}
        \label{eq:essential bound of variation}
           g(t) \leq L \ \text{for $\mathscr L^1$ almost every $t\in I$}.
        \end{equation}
        Note that for every $t\in J\sim N$, 
by \ref{Ilmanen-def}, \cite[6.6]{Ilmanen-EllipticRegularization} and \eqref{space time flow:priori local mass bound},
        the assumption \eqref{eq:essential bound of variation} holds on relatively compact subintervals $I$ of $J$.
     For every  $0\leq \omega \in \mathscr D(I)$, we estimate by 
    the definition of distributional derivative, integration by parts,
    \eqref{eq:Brakke's distributional inequality} and \eqref{eq:essential bound of variation},
        \begin{align*}\label{eq:the distributional derivative of linear support flow}
            S' (\omega) &= -v(\omega') +\textstyle \int L \cdot t\cdot  \omega'(t) \ d \mathscr L^1_t \\
            &= -v(\omega')- \textstyle \int L \cdot \omega(t) \ d \mathscr L^1_t\\
            & \leq \textstyle \int \omega(t) \cdot (g(t) - L) \ d \mathscr L^1_t \leq 0.
        \end{align*}
We may then further assume that
\begin{equation}\label{eq:finite variation of mass flow}
    \|S'\|(I)<\infty.
\end{equation}
Then there exists (see the proof of \ref{lem:one dimensional BV:1}) a unique $c\in \mathbf R$ such that
    \begin{gather*}
        u^l(t) = c + S'(\{ s : a < s < t\}) \ \text{for $t\in I$},\\
        u^r(t) = c + S'(\{ s : a < s \leq t\}) \ \text{for $t\in I$},
        \end{gather*}
     are respectively left and right continuous, with $u^{r} \leq u^{l}$, and
    \begin{gather}
        S(\omega) = \textstyle \int u^l(t)\cdot \omega(t) \ d \mathscr L^1_t \ \text{for $\omega \in \mathscr D(I)$} \label{eq:left representability} ,\\
        S(\omega) = \textstyle \int u^r(t)\cdot \omega(t) \ d \mathscr L^1_t \ \text{for $\omega \in \mathscr D(I)$}. \label{eq:right representability}
    \end{gather}
     We infer from \ref{additivity of canoical representation}, \ref{Space time Brakke flow:0}, and \eqref{eq:the differerbce distribution of flow and linear support flow} that 
     \begin{gather*}
         u^l(t) +L\cdot t \ \text{for $t\in I$},\\
         u^r(t) +L\cdot t \ \text{for $t\in I$}
     \end{gather*}
     are respectively the left and right continuous representatives of $v$, which implies that for all $t\in J$,
\begin{gather*}
    \|\langle V^+,t\rangle \|(\phi) -L\cdot t = \lim\limits_{\varepsilon \downarrow 0} \frac{1}{\varepsilon} \textstyle \int_t^{t+\varepsilon} u^r(s) \ d \mathscr L^1_s = u^r(t),\\
    \|\langle V^-,t\rangle \|(\phi) -L\cdot t= \lim\limits_{\varepsilon \downarrow 0} \frac{1}{\varepsilon} \textstyle \int_{t-\varepsilon}^{t} u^{l}(s) \ d \mathscr L^1_s= u^l(t).
\end{gather*}
    By \ref{Space time Brakke flow:0}, \eqref{eq:finite variation of mass flow} holds for every 
    relatively compact subinterval $I$ of $J$, and
     \ref{Space time Brakke flow:1}
    and \ref{Space time Brakke flow:2} follow.

Proof of \ref{Space time Brakke flow:3}, by the left continuity of $t \mapsto \| \langle V^{-}, t \rangle \| (\phi)$, and fact that $\| \langle V^{-}, t \rangle \| (\phi) = \| \langle V, t \rangle \| (\phi)$ for all $t \in N$, the inequality \eqref{space time flow:priori local mass bound} can be extended to 
every $a,b\in J$ for $\| \langle V^{-}, \, \cdot \, \rangle \| (\phi)$, and hence by \ref{def:Brakke flow}, $t \mapsto \| \langle V^{-}, t \rangle \|$ is a Brakke flow. 
By the right continuity of $t \mapsto \| \langle V^{+}, t \rangle \| (\phi)$, we similarly deduce that $t \mapsto \| \langle V^{+}, t \rangle \| (\phi)$ is a Brakke flow.
The fact that $\| \langle V^{+}, t \rangle \| (\phi) \leq \| \langle V^{-}, t \rangle \| (\phi)$
similarly follows from the left and right continuity,
\ref{Brakke flow:thm:1:3}, and the fact that these quantities are equal for $\mathscr{L}^{1}$ almost all $t \in J$.

Proof of  
\ref{Space time Brakke flow:4}, the sufficiency part of the statement follows from
 \ref{Space time Brakke flow:3}, because
\[\textstyle  \|\langle V^-,b\rangle \|(\phi) -  \|\langle V^+,a\rangle \|(\phi) 
\leq \int_a^b \mathscr B(\|\langle V,t\rangle \|,\phi) \ d \mathscr L^1_t 
\]
holds for every $a,b\in J$ with $a \leq b$, and $\phi : U \to \mathbf R$ nonnegative of class $2$ with compact support.
The necessity part of the statement follows from first noting that for $\mathscr{L}^{1}$ almost all $t \in J$, 
\begin{equation*}
    \nu (t) = \| \langle V, t \rangle \| = \| \langle V^{-}, t \rangle \| = \| \langle V^{+}, t \rangle \|,
\end{equation*}
and combining this with \ref{Brakke flow:thm:1:3}, and the left and right continuity of $\| \langle V^{-}, t \rangle \|$, and $\| \langle V^{+}, t \rangle \|$. 
To prove the postscript, we apply the argument in the proof of \ref{Space time Brakke flow:1} and \ref{Space time Brakke flow:2} to the end point of $\Clos J$ to extend the inequality \eqref{space time flow:priori local mass bound}.
    \end{proof}

\begin{definition}\label{def:Brakke flow:space time measure}
    We say a Radon measure $V$ over $J \times \mathbf{G}_{k} (U)$ is a \textit{space-time-Grassmann Brakke flow} if and only if the  following three statements hold,
    \begin{enumerate}
        \item For $\mathscr{L}^{1}$ almost all $t \in J$, $\langle V, t \rangle \in \mathbf{IV}_{k} (U)$.
        \item For every non-negative, class 2 function $\phi \colon U \rightarrow \mathbf{R}$, with compact support $\mathscr B(\|\langle V,\cdot\rangle \|,\phi) \in \mathbf{L}_{1}^{\text{loc}} (J)$.
        \item For every non-negative $\omega \in \mathscr{D} (J)$, and every non-negative, class 2 function $\phi \colon U \rightarrow \mathbf{R}$, with compact support, we have that, 
    \begin{equation*}
        - \textstyle \int \omega'(t) \phi(x) \ d \|V\|_{(t,x)} 
        \leq \textstyle \int \omega(t) \mathscr B(\|\langle V,t\rangle \|,\phi) \ d \mathscr L^1_t.
    \end{equation*}
    \end{enumerate}
\end{definition}

\begin{propparagraph}\label{proppara: Brakke flow gives a space-time-Grassmann Brakke flow}
    One obtains the following equivalence between the definitions \ref{Ilmanen-def-Brakke flow}, and \ref{def:Brakke flow:space time measure}.
    In one direction, for a space-time-Grassmann Brakke flow $V$, by \ref{Space time Brakke flow:4}, there exists a non-empty family of associated Brakke flows.
    In the other direction, given a Brakke flow $\mu \colon J \rightarrow \mathscr{M}_{+} (U)$, there exists a unique space-time-Grassmann Brakke flow $V$, associated to $\mu$, such that, 
    \begin{equation*}
        \| \langle V^{+}, t \rangle \| \leq \mu (t) \leq \| \langle V^{-}, t \rangle \|, \text{ for all $t \in J$.}
    \end{equation*}
    Indeed, we take $V$ equal to the $\lambda$ associated to $\mu$ in \ref{Brakke flow:space time measure:1}.
    Then, for a non-negative class 2 function $\phi \colon U \rightarrow \mathbf{R}$, with compact support, one applies \cite[A.2]{Lahiri2017EqualityOT} to the function $t \mapsto \mu (t) (\phi)$, making note of \ref{Brakke flow:thm:1:2}, and \ref{Brakke flow:thm:1:5}, to conclude that the function $t \mapsto \mathscr{B} (\| \langle V, t \rangle \|, \phi)$ lies in $\mathbf{L}_{1}^{\text{loc}} (J)$. 
    Moreover, for non-negative $\omega \in \mathscr{D} (J)$, the distributional Brakke inequality in \ref{def:Brakke flow:space time measure} follows from applying \cite[A.2]{Lahiri2017EqualityOT} to the function $t \mapsto \omega (t) \mu (t) (\phi)$, and
    \begin{equation}\label{eq:partial space time ineq}
        \overline{\D}_{t} (\omega (t) \mu (t) (\phi))|_{s} \leq \omega' (s) \mu (s) (\phi) + \omega (s) \mathscr{B} (\mu (s), \phi) \quad \text{for $s\in J$}.
    \end{equation}
    To see \eqref{eq:partial space time ineq} , we may assume $-\infty < \D^+_t (\omega (t) \mu (t) (\phi))|_{s}$, then  by \ref{Brakke flow:thm:1:3-add-add}, 
    \begin{gather*}
        \lim\limits_{s\downarrow t} \mu(s)(\omega'(t)\cdot \phi) = \mu(t)(\omega'(t)\cdot\phi),
    \end{gather*}
Apply $\limsup\limits$ as $s \downarrow t$ to
    \[\mu(s)\left[\left(\frac{\omega(s)-\omega(t)}{s-t}- \omega'(t)\right)\cdot \phi\right] + \mu(s)(\omega'(t)\phi) + \left(\frac{\mu(s)-\mu(t)}{s-t}\right)(\omega(t)\cdot \phi),\]
    taking \ref{Brakke flow:thm:1:2} into account, and
     applying a similar argument for $\D^-_{t} (\omega (t) \mu (t) (\phi))|_{s}$ we obtain  the assertion.
    
    The uniqueness of $V$ then follows from combining the fact that for any space-time-Grassmann Brakke flow $W$, by \ref{Space time Brakke flow:0},
    \begin{equation*}
        W_{(t,x, S)}(\omega(t)\cdot \alpha(x, S)) = \textstyle \int \omega(t)\cdot \langle W, t\rangle (\alpha) \ d \mathscr L^1_t,
    \end{equation*}
    for every $\alpha \in \mathscr{K} (\mathbf{G}_{k} (U))$, and $\omega \in \mathscr{K} (J)$, along with the fact that for $\mathscr{L}^{1}$ almost all $t \in J$, $\langle W, t \rangle \in \mathbf{IV}_{k} (U)$, and thus will be uniquely determined by its weight $\| \langle W, t \rangle \|$.
    The claimed inequalities now follow directly from \ref{Space time Brakke flow:4}.
\end{propparagraph}

\begin{remark}
    For a space-time-Grassmann Brakke flow $V$, as 
    non-negaitve members of $\mathscr{D} (J \times U)$ can be approximated by a non-negative sequence of
    $\mathscr{D} (J) \otimes \mathscr{D} (U)$ (see \cite[3.1]{Menne-WDFV} or \cite[4.1.2, 4.1.3]{MR41:1976}), and by \ref{Brakke flow:thm:1:4}, and the fact that $\mathbf{h} (\langle V, t \rangle, x)$ is $\| V \|$ measurable (from \ref{thm:measurable function:tangent,density,mean curvature vector}), we may rewrite the Brakke inequality as the following 
    \begin{eqnarray*}
        && \textstyle \int \D_{s} \phi (s, x)|_{t} \, d \| V \|_{(t, x)} \\
        && \hspace{1cm} \leq \textstyle\int - \phi(t, x) |\mathbf{h} (\langle V, t \rangle, x)|^{2} + \mathrm{grad}_{y} \phi (t, y)|_{x} \bullet \mathbf{h} (\langle V, t \rangle, x) \, d \| V \|_{(t, x)},
    \end{eqnarray*}
    for any non-negative $\phi \in \mathscr{D} (J \times U)$.
\end{remark}
\begin{remark}
    Suppose $J$ is a bounded open interval, $\nu $ is a Brakke flow from $\Clos J$ to  set of Radon measures over $U$. Then 
    for every nonnegative $\phi : U \to \mathbf R$ of class $2$ with compact support, 
    the function  
    \[ v(t) = \nu(t)(\phi) \ \text{for $t\in J$}\]
    satisfies $\|v\|(J) + \|v'\|(J)<\infty$.
     So the condition \eqref{sufficient condition for extension} is also necessary for the extension of Brakke flow to exist.

    In fact,
    by \eqref{Brakke flow variation bound} and \cite[A.1]{Lahiri2017EqualityOT}, there exists a positive real number $L$ such that
    \[ v(t) - L\cdot t \quad \text{for $t\in J$}\] is a bounded nonincreasing function, and
    the assertion then follows from \cite[3.27]{ambrosio2000functions}.
    
\end{remark}
\bibliographystyle{plain}

\Addresses

\end{document}